\documentclass[a4paper,11pt,reqno]{amsart}
\usepackage[utf8]{inputenc}
\usepackage[english]{babel}
\usepackage{amsmath}
\usepackage{amsfonts}
\usepackage{amsthm}
\usepackage{mathtools}
\usepackage{float}
\usepackage{stmaryrd}
\usepackage{graphics}
\usepackage{graphicx}
\usepackage{subfig}
\usepackage{enumitem}
\usepackage{datetime}
\usepackage{comment}
\usepackage{todonotes}
\usepackage{bm}
\usepackage[colorlinks=false]{hyperref}
\usepackage{color}

\newcommand{\mc}{\mathcal}

\renewcommand{\d}{\,\mathrm{d}}
\newcommand{\tr}{\mathrm{tr}}
\newcommand{\dom}{\mathrm{dom}\,}
\renewcommand{\div}{\mathrm{div}\,}

\DeclareMathOperator*{\dist}{dist}
\DeclareMathOperator*{\sym}{sym}

\newcommand\R{\mathbb R}

\renewcommand{\to}{\rightarrow}

\numberwithin{equation}{section}
\newtheorem{thm}{Theorem}[section]

\newtheorem{prop}[thm]{Proposition}

\theoremstyle{definition}
\begingroup
\newtheorem{rmk}[thm]{Remark}

\endgroup

\theoremstyle{remark}

\title[Dimension reduction with pressure loads]{The effects of pressure  loads in the dimension reduction of elasticity models}

\author[M. Kru\v z\'ik and F. Riva]{Martin Kru\v z\'ik and Filippo Riva}

\address[M. Kru\v z\'ik]{Institute of Information Theory and Automation, Czech Academy of Sciences, Pod vod\'arenskou ve\v z\'i 4, CZ-182 00, Prague 8, Czech Republic}
\email{kruzik@utia.cas.cz}
\address[F. Riva]{Dipartimento di Scienze delle Decisioni, Universit\`{a} Commerciale Luigi Bocconi, via Roentgen 1, 20136 Milano, Italy}
\email{filippo.riva@unibocconi.it}

\begin{document}

	\begin{abstract}
		
		We study the dimensional reduction from three to two dimensions in hyperelastic materials subject to a live load, modeled as a constant pressure force. Our results demonstrate that this loading has a significant impact in higher-order scaling regimes, namely those associated with von Kármán-type theories, where a nontrivial interplay arises between the elastic energy and the pressure term. In contrast, we rigorously show that in lower-order bending regimes, as described by Kirchhoff-type theories, the pressure load does not influence the minimizers. Finally, after identifying the corresponding $\Gamma$-limit, we conjecture that a similar independence from the pressure term persists in the most flexible membrane regimes.  
	\end{abstract}
	
	\maketitle
	
	{\small
		\par
	}
	
	\pagenumbering{arabic}
	
	\medskip
	
	\tableofcontents
	
	\section*{Introduction}
	Dimension reduction represents an important class of problems in mathematical analysis, numerics, and engineering. In elasticity, it is often used to approximate the behavior of elastic bodies by reducing their dimensionality, passing from three dimensions (3D) to lower-dimensional models, such as two-dimensional (2D) or one-dimensional (1D), based on suitable assumptions about the geometry and the deformation of the structure. For example, thin plates can be approximated by 2D models, while slender beams can be reduced to 1D objects. These reduced models are not only much more computationally efficient but also provide valuable insights into the underlying physics of the system.  The key point is that  the lower-dimensional structures are derived from the full-dimensional models in a rigorous way. For static problems, this is usually done by means of $\Gamma$-convergence \cite{braides,dalmaso}. 
	
	Regarding bidimensional systems, starting from the celebrated works \cite{LeDretRaoult, friesecke.james.mueller1, friesecke.james.mueller2} where a whole hierarchy of plate models has been derived from nonlinear elasticity, the analysis has been subsequently extended in many different directions. For instance, the current understanding encompasses prestrained \cite{padillagarza}, magnetic \cite{bresciani}, brittle \cite{almi.reggiani.solombrino}, or martensitic \cite{Tolottimulti} materials. Moreover, in an evolutive framework, dimension reduction for dynamic models has been investigated in \cite{AbelsMoraMuller}, while viscoelastic or thermoviscoelastic effects are taken into account in \cite{friedkruz} and \cite{Badfriedmach}, respectively. We finally mention \cite{lecumberry.mueller} and the recent paper \cite{Tolotti}, where the issue of stability of thin bodies with respect to volume forces is addressed. 
	
	Despite such extensive developments, in most cases, external loads are considered to be independent of deformations, which significantly simplifies the mathematical discussion, but is definitely a rare situation in practical applications. Such forces are commonly termed \emph{dead loads}, in contrast to the more mechanically relevant ones, the so-called \emph{live loads} \cite{ciarlet, PodGuid}, which actually depend on the deformed configuration of the body. The technical advantage of considering dead loads over live ones lies in their simpler mathematical structure: dead loads are typically represented by linear bounded functionals, whereas live loads involve more intricate, nonlinear formulations. For this reason, up to our best knowledge, variational asymptotic analysis of models including a live load has been performed just recently in \cite{MoraRiva} in the framework of linearization problems. Therein, the authors focus on boundary pressure forces, starting from the observations raised in \cite{MaorMora} where Neumann boundary conditions, which share some similarities with pressure, are considered. However, we are not aware of dimension reduction results for models influenced by live loads.
	
	In this contribution, we provide rigorous derivations of 2D elasticity theories, in the spirit of \cite{friesecke.james.mueller2}, encompassing constant-pressure live loads on the whole Neumann part of the boundary. Such setting may for instance simulate a small thin elastic body completely immersed in a fluid, hence experiencing hydrostatic pressure. Considering a more general non-constant pressure (as in the linearization procedure developed in \cite{MoraRiva}) would pose serious additional difficulties in the dimension reduction analysis, so we prefer to postpone such a case to future works and to limit ourselves to the constant pressure scenario, which still provides interesting and unexpected outcomes. Note also that applying a pressure force only on a {\it part} of the Neumann boundary  is a non-potential load and therefore cannot be treated in a variational setting. These problems lead to the so-called flutter phenomenon, see, e.g., \cite{shubov}.
	
	Our starting point is a three-dimensional energy functional 
	\begin{equation*}
		I_h^{\alpha,\pi}(w)=
		\int_{\Omega_h} W(\nabla w(x))\d x +h^\alpha\pi\int_{\Omega_h}\det\nabla w(x)\d x,
	\end{equation*}
	where $\Omega_h=S\times(-h/2,h/2)$ is a thin cylinder with mid-plane $S\subseteq \R^2$ and thickness $h>0$, representing the reference configuration of a hyperelastic body, while $w\colon \Omega_h\to \mathbb{R}^3$ denotes its deformation. The first integral above describes the nonlinear elastic energy of the material, while the second integral is the potential of a constant applied pressure load \cite{ciarlet}, with intensity $h^\alpha\pi\in \R$, acting on the boundary of the deformed configuration $w(\Omega_h)$. 
	Roughly speaking, the limiting planar model is described by the $\Gamma$-limit of the rescaled energy ${h^{-1-2\alpha}}I_h^{\alpha,\pi}$ as the width $h$ of the material tends to zero. The parameter $\alpha\ge 0$, which tunes the intensity of pressure, is responsible for different resulting two-dimensional theories, as it happens in the known pressureless cases. The case $\alpha=0$ corresponds to the so-called membrane regime \cite{LeDretRaoult}, while $\alpha=1$ and $\alpha\in(1,2)$ lead to the nonlinear \cite{friesecke.james.mueller1} and to the linearized bending regime \cite{friesecke.james.mueller2}, respectively. The higher scalings $\alpha=2$ and $\alpha>2$ finally describe the von K\'arm\'an and the linearized von K\'arm\'an regime \cite{friesecke.james.mueller2}, respectively.
	
	Our results show that the presence of pressure surprisingly affects only the latter more rigid scalings, while the flexible ones behave exactly as if pressure was not acting, in terms of minimizers. In particular, for $\alpha\ge 2$ we prove that, with respect to a suitable topology, the $\Gamma$-limit
	\begin{equation*}
		\frac{1}{h^{1+2\alpha}}I_h^{\alpha,\pi}\xrightarrow[h\to 0]{\Gamma}\mc E^\alpha_{\rm vK}+\pi \mathbb Q^\alpha+\pi^2c^\alpha,
	\end{equation*}
	holds true, where $\mc E^\alpha_{\rm vK}$ is the known von K\'arm\'an functional (linearized for $\alpha>2$) in absence of pressure, depending on the horizontal and vertical displacements $(u,v)\colon S\to \mathbb R^2\times\mathbb R$ of the plate, $c^\alpha$ is a constant, while $\mathbb  Q^\alpha$ is a nontrivial quadratic form directly connected to the elastic energy density $W$ (see the term multiplied by $\pi$ in \eqref{eq:vKenergypi} and \eqref{eq:vKlinenergypi}). In particular, minimizers of the limit energy are influenced by the pressure load via the operator $\mathbb Q^\alpha$, and so they differ from the ones of the pure von K\'arm\'an energy $\mc E^\alpha_{\rm vK}$. For instance, in the specific case of homogeneous isotropic materials, described in a linearized setting via the Lam\'e coefficients $\mu$ and $\lambda$, we can explicitly compute
	
	\begin{equation*}
		\mathbb Q^\alpha(u,v)=\begin{cases}\displaystyle
			\frac{2\mu}{2\mu+\lambda}\int_S\left(\div' u(x')+\frac 12|\nabla' v(x')|^2\right)\d x',&\text{if }\alpha=2,\\\displaystyle
			\frac{2\mu}{2\mu+\lambda}\int_S\div' u(x')\d x',&\text{if }\alpha>2.
		\end{cases}
	\end{equation*}
	On the contrary, for $\alpha\in[1,2)$ we show that
	\begin{equation*}
		\frac{1}{h^{1+2\alpha}}I_h^{\alpha,\pi}\xrightarrow[h\to 0]{\Gamma}\mc E^\alpha_{\rm ben}(u,v)+c_\pi,
	\end{equation*}
	where $\mc E^\alpha_{\rm ben}$ is the known bending energy (linearized for $\alpha\in(1,2)$) of Kirchoff theory in absence of pressure, and $c_\pi$ is a constant. This implies that, in the limit, equilibrium configurations in presence of pressure precisely coincide with the ones of the pressureless framework.
	
	Finally, in the membrane regime $\alpha=0$, we prove that
	\begin{equation*}
		\frac{1}{h}I_h^{0,\pi}\xrightarrow[h\to 0]{\Gamma}  \mc E^\pi_{\rm mem}(y)=\begin{cases}
			\displaystyle\int_S\mc Q\big(W(\cdot)+\pi\det(\cdot)\big)_0(\nabla' y(x'))\d x', &\text{ if }\partial_3y=0\\
			+\infty, & \text{ otherwise,}
		\end{cases}
	\end{equation*}
	where $\mc Q$ denotes the quasiconvex envelope, while the subscript $0$ indicates the procedure of minimizing with respect the third column. Inspired by the previous considerations which suggest that pressure effects should just be visible at high scalings, we conjecture that
	\begin{equation}\label{eq:conjecture}
		\mc Q\big(W(\cdot)+\pi\det(\cdot)\big)_0= \mc Q\big(W)_0+c_\pi,
	\end{equation}
	so that minimizers of $\mc E^\pi_{\rm mem}$ are actually not affected by $\pi$. Although we are not able to prove the validity of such conjecture, in Section \ref{sec:examplemembrane} we propose an explicit example which indicates its reliability, at least for positive pressure $\pi>0$.
	
	The paper is organized as follows. Section \ref{sec:setting} presents in details the three-dimensional hyperelastic model and fixes the main assumptions. After a brief recall of the known dimension reduction results in the pressureless framework $\pi=0$, we state our findings in Theorem \ref{thm:main}, where all the different regimes are considered. In Section \ref{sec:membrane} we analyze the membrane regime $\alpha=0$ and we state our conjecture \eqref{eq:conjecture}. Section \ref{sec:highscalings} is finally devoted to the proof of the $\Gamma$-convergence results both in the bending ($\alpha\in [1,2)$) and in the von K\'arm\'an-type ($\alpha\ge 2$) regimes.
	
	\subsection*{Notations}
	
	We adopt standard notations for Lebesgue and Sobolev spaces. 
	Throughout the paper, the symbols $C$ or $c$ will be used to indicate some positive constants not depending of $h$, whose value may change from line to line.
	
	The maximum (resp. minimum) of two extended real numbers $\alpha,\beta\in \R\cup\{\pm\infty\}$ is denoted by $\alpha\vee\beta$ (resp. $\alpha\wedge\beta$). Without risk of ambiguity, by $\xi\wedge\zeta\in \R^3$ we also mean the cross (or wedge) product between the vectors $\xi\in \R^3$ and $\zeta\in \R^3$. The scalar product between $\xi$ and $\zeta$ is instead indicated by $\xi\cdot\zeta$.
	
	We denote by $\R^{n\times n}$ and $\R^{n\times n}_{\rm sym}$ the set of $n\times n$ matrices and its subset of symmetric matrices. Given a matrix $G\in \R^{n\times n}$, its symmetric part $(G+G^T)/2$ is indicated by $\operatorname{sym} G$. The set of rotations is denoted by $SO(n)$, namely
	\begin{equation*}
		SO(n)=\{R\in\R^{n\times n}: \ R^TR=I, \ \det R=1\}.
	\end{equation*}
	Given a matrix $F\in \R^{3\times 3}$, we will often write $F=(F_1|F_2|F_3)$, where $F_i$ denotes the $i$-th column of $F$.
	
	We finally recall that for every $F=(f_{ij})_{i,j}\in \R^{3\times 3}$ and $h>0$ the following expansion of the determinant holds
	\begin{equation}\label{eq:sviluppodet}
		\det(I+h F)=1+h\tr F+h^2\iota_2(F)+h^3\det F,
	\end{equation}
	where 
	\begin{equation}\label{eq:iota}
		\iota_2(F)=f_{11}f_{22}-f_{12}f_{21}+f_{11}f_{33}-f_{13}f_{31}+f_{22}f_{33}-f_{23}f_{32}.
	\end{equation}	
	\section{Setting and main results}\label{sec:setting}
	
	Let $\Omega_h=S\times(-h/2,h/2)$ be the reference configuration of a thin three-dimensional elastic body of width $h\in(0,1]$, where $S$ is a bounded Lipschitz domain in $\R^2$. For lightness of notations, we will write $\Omega$ in place of $\Omega_1$.
	
	By varying the parameter $\alpha\ge 0 $, which will modulate the intensity of the applied pressure, let us consider the total energy functionals $I_h^{\alpha,\pi}\colon W^{1,2}(\Omega_h;\R^3)\to (-\infty,+\infty]$ defined as
	\begin{equation}\label{eq:I}
		I_h^{\alpha,\pi}(w)=\begin{cases}\displaystyle
			\int_{\Omega_h} W(\nabla w(x))\d x +h^\alpha\pi\int_{\Omega_h}\det\nabla w(x)\d x,& \text{ if }w\in \dom I_h^{\alpha,\pi},\\
			+\infty, &\text{otherwise,}
		\end{cases}
	\end{equation}
	whose domain $\dom I_h^{\alpha,\pi}$ consists of those deformations $w\in  W^{1,2}(\Omega_h;\R^3)$ such that both $W(\nabla w)$ and the Jacobian $\det\nabla w$ are summable in $\Omega_h$. The first integral above represents the elastic bulk energy stored in the material, and the function $W\colon \R^{3\times 3}\to (-\infty,+\infty]$ is the stored energy density. On the other hand, the second integral is the potential of a constant applied pressure load, with intensity $h^\alpha\pi\in \R$, acting on the boundary of the deformed configuration. Indeed, it is well known \cite{ciarlet,PodGuid} that this term is a null Lagrangian, namely it does not alter the Euler-Lagrange equations of \eqref{eq:I} which still read as 
	\begin{equation*}
		-\operatorname{div}\partial_F W(\nabla w(x))=0,\qquad \text{in }\Omega_h,
	\end{equation*}
	but it affects their boundary conditions, which can be written in the reference configuration as
	\begin{equation*}
		\partial_F W(\nabla w(x))n_{\partial \Omega_h}(x)=-h^\alpha \pi \,{\rm cof}\,\nabla w(x)n_{\partial \Omega_h}(x), \qquad\text{for }x\in \partial\Omega_h,
	\end{equation*}
	where ${\rm cof}\,F$ denotes the cofactor of the matrix $F$ and $n_{\partial\Omega_h}$ is the outward unit normal to $\partial\Omega_h$. Notice that in the deformed configuration the right-hand side above reads as
	\begin{equation}\label{eq:bdrycond}
		-h^\alpha \pi n_{\partial(w(\Omega_h))}(z),\qquad\text{for }z\in \partial(w(\Omega_h)),
	\end{equation}
	namely it represents exactly a pressure force with intensity $h^\alpha\pi$ acting on  $\partial(w(\Omega_h))$ in the normal direction. Since the behaviour of such loads depends on the deformed configuration, we recall that pressure forces fall within the class of so-called \emph{live loads} \cite{PodGuid}.
	
	As it is customary in dimension reduction problems, it is useful to reformulate the system in a fixed domain. We thus perform the change of variables
	\begin{equation*}
		{id}_h:\Omega\to\Omega_h,\quad{id}_h(x):=(x_1,x_2,hx_3),
	\end{equation*}
	and, setting $y:=w\circ {id}_h$, in order to normalize volume effects we consider the rescaled energy $\mc E^{\alpha,\pi}_h\colon W^{1,2}(\Omega;\R^3)\to (-\infty,+\infty] $ defined as
	\begin{equation}\label{eq:energy}
		\mc E^{\alpha,\pi}_h(y):=\frac{1}{h}I^{\alpha,\pi}_h(w)=\begin{cases}\displaystyle
			\int_{\Omega}W(\nabla_h y(x)) \d x +h^\alpha\pi\int_\Omega\det\nabla_h y(x)\d x,& \text{if } y\in \dom \mc E^{\alpha,\pi}_h,\\
			+\infty, &\text{otherwise,}
		\end{cases}
	\end{equation}
	where $\dom \mc E^{\alpha,\pi}_h$ is defined analogously to $\dom I^{\alpha,\pi}_h$, with the obvious changes. Above, the rescaled gradient $\nabla_h$ denotes the differential operator defined as
	\begin{equation*}
		\nabla_h y :=\left(\partial_1 y \big|\partial_2 y \big|\frac{\partial_3 y }{h}\right).
	\end{equation*}

	In order to perform the asymptotic analysis of \eqref{eq:energy} as $h\to 0$, we require that the density $W$ satisfies the following standard properties for all $F\in \R^{3\times 3}$:
	\begin{subequations}\label{hyp:bending}
		\begin{align}
			&\bullet\, W(RF)=W(F)\text{ for all }R\in SO(3);\quad \text{ (frame indifference)} \\
			&\bullet\, W(F)=0\quad\iff\quad F\in SO(3); \quad \text{ (stress-free reference configuration)}\label{eq:minSO3}\\
			&\bullet\, W(F)\ge  c \dist(F;SO(3))^2 \quad \text{ (coercivity)}\label{hyp:Wdist}\\			
			&\bullet\, W \text{ is of class $C^2$ in a neighborhood of $SO(3)$ with bounded second derivatives.}
		\end{align}
	\end{subequations}
	We also recall that the growth condition \eqref{hyp:Wdist} implies
	\begin{equation}\label{eq:growthdet}
		W(F)\ge c|\det F-1|^2,\text{ whenever }|\det F-1|\le 1.
	\end{equation}
	
	Additionally, in order to handle the pressure term in the energy functional, we need to assume either positivity of $\pi$ together with the orientation preserving condition of $W$, namely
	\begin{equation}\label{eq:orientationpreserving}
		\pi>0 \quad\text{ and }\quad W(F)=+\infty\text{ if } \det F\le 0,\\
	\end{equation}
	or a linear control of $W$ on the determinant, completing \eqref{eq:growthdet}, of the form
	\begin{equation}\label{hyp:pineg}
		W(F)\ge c|\det F-1|,\text{ whenever }|\det F-1|>1.
	\end{equation}
	We point out that both the above conditions on $W$ are very natural and physically relevant, and often satisfied at the same time. However, from the mathematical point of view, the orientation preserving condition alone is enough to handle a positive pressure, while for negative pressures the control on the determinant \eqref{hyp:pineg} by the density $W$ is crucial. This will be clear from the proof of the compactness result stated in Proposition~\ref{prop:bounds}.    
	
	As in the classical case with no pressure ($\pi=0$), different assumptions on $W$ are needed in the so-called membrane regime $\alpha=0$ \cite{Hafsa.Mandallena1,Hafsa.Mandallena2,LeDretRaoult}. In this framework, we assume that $W$ is continuous and that for all $F\in \R^{3\times 3}$ there hold 
	\begin{subequations}\label{hyp:membrane}
		\begin{align}
			&\bullet\, W(F)=+\infty\quad\iff\quad \det F\le 0;\label{hyp:orientation}\\
			&\bullet\, W(F)\ge c_1|F|^p-c_2,\text{ for some $p>1$ and $c_1,c_2>0$;}\label{hyp:Wgrowth}\\			
			&\bullet\, \text{for all $\delta>0$ there exists $C_\delta>0$ such that $W(F)\le C_\delta(1+|F|^p)$ if $\det F\ge \delta$}.\label{hyp:pbounddelta}
		\end{align}
	\end{subequations}
	\begin{rmk}
		Hypothesis \eqref{hyp:pbounddelta} has been introduced in \cite{Hafsa.Mandallena2} as a weak growth condition, compatible with the orientation preserving condition \eqref{hyp:orientation}, for which the dimension reduction is still feasible. It improves the uniform growth condition required in the pioneering work \cite{LeDretRaoult}, which is actually unphysical. For this reason, in the current paper we prefer to stick to assumptions \eqref{hyp:membrane} for the membrane regime, although the presence of pressure could also be handled in the framework of \cite{LeDretRaoult}, with small adjustments.
	\end{rmk}

	In the regimes $\alpha>0$ (when \eqref{eq:minSO3} is in force), in order to catch the correct rescaling of minimizing sequences it will also be useful to introduce the functionals $\overline{\mc E}^{\alpha,\pi}_h(y):=\mc E^{\alpha,\pi}_h(y)-\mc E^{\alpha,\pi}_h({id}_h)$, namely
	\begin{equation}\label{eq:energybar}
		\overline{\mc E}^{\alpha,\pi}_h(y)= \begin{cases}\displaystyle
			\int_{\Omega}W(\nabla_h y(x)) \d x +h^\alpha\pi\int_\Omega\Big(\det\nabla_h y(x)-1\Big)\d x,& \text{if }{ y\in \dom \mc E^{\alpha,\pi}_h},\\
			+\infty, &\text{otherwise.}
		\end{cases}
	\end{equation}
	This operation clearly does not affect minimizers (nor almost-minimizers) of $\mc E^{\alpha,\pi}_h$.
	
	\subsection{Known pressureless case}
	Before presenting our results, we review the known framework with no pressure ($\pi=0$). For lightness of exposition, we will write $\mc E^{\alpha}_h$ in place of $\mc E^{\alpha,0}_h$. 
	
	It will also be convenient to introduce the following notations, which we will adopt throughout the whole paper. For $x\in\Omega$, we write $x=(x',x_3)\in S\times\left(-\frac 12,\frac 12\right)$; similarly, for $\R^3$-valued functions $y$ we denote by $y'=(y_1,y_2)$ their first two components. Moreover, the symbols $\nabla'$ and $(\nabla')^2$ stand for the gradient and the Hessian with respect to $x'$, i.e. $\nabla' y=(\partial_1 y,\partial_2 y)$ and $(\nabla')^2y=(\nabla'\partial_1 y,\nabla'\partial_2 y)$. The symmetric gradient with respect to $x'$ of a $\R^2$-valued function $u$ will be instead denoted by $e'(u)$.
	
	We introduce the quadratic forms $Q_3$ and $Q_2$, defined respectively on $\R^{3\times 3}$ and $\R^{2\times 2}$ as
	\begin{equation}\label{def:Q2}
		Q_3(F):=D^2W(I)[F]^2,\qquad Q_2(G):=\min\limits_{a\in \R^3}Q_3(G+a\otimes e_3),
	\end{equation}
	where, given a matrix $G\in \R^{2\times 2}$, with a slight abuse of notation we still write $G$ for the ${3\times 3}$ matrix $\begin{pmatrix}
		\begin{array}{c|c}
			G & 0 \\
			\hline 0 & 0
		\end{array}
	\end{pmatrix}$. Conversely, given $G\in \R^{3\times 3}$, we denote by $G_{2\times 2}$ the $2\times 2$ matrix obtained by removing the third column and row of $G$.
	
	We first recall the following compactness result, which we will exploit also when a  pressure term  is present.
	
	\begin{prop}[Compactness \cite{friesecke.james.mueller1,friesecke.james.mueller2}]\label{prop:compactness}
		Assume that $W$ satisfies \eqref{hyp:bending} and let $y_h$ be such that ${\mc E}^\alpha_h(y_h)\le C h^{2\alpha}$, for some $\alpha\ge 1$. Then the following facts hold:
		\begin{itemize}
			\item If $\alpha=1$, up to a nonrelabelled subsequence, one has
			\begin{equation}\label{eq:gradconv}
				\nabla_h y_h \xrightarrow[h\to 0]{}(\nabla' y|b)\in W^{1,2}(S;\R^{3\times 3}),\qquad\text{strongly in }L^2(\Omega;\R^{3\times 3}),
			\end{equation}
			where the normal vector $b$ is defined as
			\begin{equation}\label{eq:normalvector}
				b(x'):= \partial_1 y(x')\wedge \partial_2 y(x').
			\end{equation}
			In particular, $(\nabla' y(x')|b(x'))$ belongs to $SO(3)$ for almost every $x'\in S$.
			
			Moreover, there exist a subset $S_h\subseteq S$ and rotations $R_h\colon S_h\to SO(3)$ such that, after setting
			\begin{equation}\label{eq:Gh}
				G_h(x):=\begin{cases}\displaystyle
					\frac{R_h(x')^T\nabla_h y_h(x)-I}{h},&\text{if } x\in S_h\times \left(-\frac 12,\frac 12\right),\\
					0,&\text{otherwise},
				\end{cases}
			\end{equation}
			one has (up to subsequences)
			\begin{equation}\label{eq:Ghconv}
				G_h\xrightharpoonup[h\to 0]{} G,\qquad\text{weakly in }L^2(\Omega;\R^{3\times3}),
			\end{equation}
			and 
			\begin{equation}\label{eq:setvanish}
				\lim\limits_{h\to 0}\left|\Omega\setminus \left(S_h\times\left(-\frac 12,\frac 12\right)\right)\right|=0.
			\end{equation}
			Finally, one can write
			\begin{equation}\label{eq:Gdecomp}
				G_{2\times 2}(x)= G_0(x')+x_3II(x'),
			\end{equation}
			where $G_0\in L^2(S;\R^{2\times 2})$ and $II(x')$ denotes the second fundamental form
			\begin{equation}\label{eq:IIform}
				II(x'):=\nabla' y(x')^T\nabla' b(x').
			\end{equation}.
			
			\item If $\alpha >1$, there exist rotations $R_h\colon S\to SO(3)$, constant rotations $\overline R _h\in SO(3)$ and constant vectors $c_h\in \R^3$ such that, after defining
			\begin{subequations}\label{eq:defuv}
				\begin{align}
					&\widetilde{y}_h(x):=\overline R_h^T y_h(x)-c_h, \qquad\qquad\qquad G_h(x):=\frac{R_h(x')^T\nabla_h \widetilde{y}_h(x)-I}{h^\alpha},\label{eq:Gh2}\\
					& u_h(x'):=\left(\frac{1}{h^{2(\alpha-1)}}\wedge\frac{1}{h^\alpha}\right)\int_{-\frac 12}^{\frac 12}(\widetilde{y}_h'(x',x_3)-x')\d x_3, \\
					& v_h(x'):=\frac{1}{h^{\alpha-1}}\int_{-\frac 12}^{\frac 12}(\widetilde{y}_h)_3(x',x_3)\d x_3,
				\end{align}
			\end{subequations}
			up to a nonrelabelled subsequence one has
			\begin{subequations}\label{eq:convuv}
				\begin{align}
					&\nabla_h\widetilde y_h\xrightarrow[h\to 0]{} I,&&\text{strongly in }L^{2}(\Omega;\R^{3\times 3}),\label{eq:convuva}\\
					&u_h\xrightharpoonup[h\to 0]{} u,&&\text{weakly in }W^{1,2}(S;\R^2), \label{eq:convuvb}\\
					&v_h\xrightarrow[h\to 0]{} v\in W^{2,2}(S),&&\text{strongly in }W^{1,2}(S),\label{eq:convuvc}\\
					&G_h\xrightharpoonup[h\to 0]{} G,&&\text{weakly in }L^{2}(\Omega;\R^{3\times 3}).\label{eq:Ghconv2}
				\end{align}
			\end{subequations}
			Moreover, one can write
			\begin{equation}\label{eq:Gv}
				G_{2\times 2}(x)= G_0(x')-x_3(\nabla')^2 v(x'),
			\end{equation}
			and the following equalities hold:
			\begin{subequations}
				\begin{align}
					& \left|e'(u(x'))+\frac 12 \nabla' v(x')\otimes \nabla' v(x')\right|=\det (\nabla')^2 v(x')=0, &&\text{if }\alpha\in(1,2),\label{eq:linconstr}\\
					& \sym G_0(x')=e'(u(x'))+\frac 12 \nabla' v(x')\otimes \nabla' v(x'), &&\text{if }\alpha=2,\label{eq:G0vK}\\
					& \sym G_0(x')=e'(u(x')), &&\text{if }\alpha>2. \label{eq:G0vKlin}
				\end{align}
			\end{subequations}
		\end{itemize}
	\end{prop}

	Since under the solely assumptions \eqref{hyp:bending} (or \eqref{hyp:membrane}) the energy $\mc E^{\alpha}_h$ may have no minimizers, it is useful to consider the following notion of almost-minimizers. We say that $y_h$ is an $\alpha$-minimizer of $\mc E^\alpha_h$ if
	\begin{equation*}
		\lim\limits_{h\to 0}\frac{1}{h^{2\alpha}}\Big(\mc E^\alpha_h(y_h)-\inf \mc E^\alpha_h\Big)=0.
	\end{equation*}
	The asymptotics of $\mc E^\alpha_h$ can be then stated as follows.
	
	\begin{prop}[]\label{prop:known}
		{\rm\bf [Membrane theory ($\boldsymbol{\alpha=0}$) \cite{belgacem,Hafsa.Mandallena1,Hafsa.Mandallena2,LeDretRaoult}]} Assume that $W$ is continuous and  satisfies \eqref{hyp:membrane}. Let $y_h$ be a $0$-minimizer of $\mc E^{0}_h$. Then there exist constant vectors $c_h\in \R^3$ such that, up to (non relabelled) subsequences, $y_h-c_h$ converges as $h\to 0$ in the weak topology of $W^{1,p}(\Omega;\R^3)$ to a minimizer of the membrane energy $\mc E_{\rm mem}\colon L^p(\Omega;\R^3)\to (-\infty,+\infty]$, defined as
		\begin{equation}\label{eq:membraneenergy}
			\mc E_{\rm mem}(y):=\begin{cases}
				\displaystyle\int_S\mc QW_0(\nabla' y(x'))\d x', &\text{ if } y\in W^{1,p}(\Omega;\R^3) \text{ and }\partial_3y=0,\\
				+\infty, & \text{ otherwise.}
			\end{cases}
		\end{equation}
		Moreover, $\partial_3 y_h$ vanishes as $h\to 0$ in the strong topology of $L^p(\Omega;\R^3)$.
		
		In \eqref{eq:membraneenergy}, the function $W_0: \R^{3\times 2}\to (-\infty,+\infty]$ is defined as 
		\begin{equation*}
			W_0(\xi):=\inf\limits_{a\in\R^3} W(\xi|a),
		\end{equation*}
		while $\mc QW_0$ denotes the quasiconvex envelope \cite{Dacorogna} of $W_0$.

		{\rm\bf [Nonlinear bending theory ($\boldsymbol{\alpha=1}$) \cite{friesecke.james.mueller1}]} Assume that $W$ satisfies \eqref{hyp:bending} and let $y_h$ be a $1$-minimizer of $\mc E^{1}_h$. Then  there exist constant vectors $c_h\in \R^3$ such that, up to (non relabelled) subsequences, $y_h-c_h$ converges as $h\to 0$ in the strong topology of $W^{1,2}(\Omega;\R^3)$ to a minimizer of the bending energy $\mc E_{\rm ben}\colon W^{1,2}(\Omega;\R^3)\to [0,+\infty]$, defined as
		\begin{equation}\label{eq:bendingenergy}
			\mc E_{\rm ben}(y):=\begin{cases}
				\displaystyle\frac{1}{24}\int_SQ_2(II(x'))\d x', &\text{ if }\partial_3y=0  \text{ and }y\in W^{2,2}_{\rm iso}(S;\R^3),\\
				+\infty, & \text{ otherwise,}
			\end{cases}
		\end{equation}
		where the second fundamental form $II(x')$ has been introduced in \eqref{eq:IIform}. Moreover, \eqref{eq:gradconv}, \eqref{eq:Ghconv} and \eqref{eq:Gdecomp} hold true.
		
		In \eqref{eq:bendingenergy}, the set of Sobolev isometries is defined as
		\begin{equation*}
			W^{2,2}_{\rm iso}(S;\R^3):=\{y\in W^{2,2}(S;\R^3):\, \nabla' y(x')^T \nabla' y(x')=I \}.
		\end{equation*}

		{\rm\bf [Linearized bending theory ($\boldsymbol{\alpha\in(1,2)}$) \cite{friesecke.james.mueller2}]}
		Assume that $W$ satisfies \eqref{hyp:bending} and let $S$ be simply connected if $\alpha\in (1,3/2)$.  Let $y_h$ be an $\alpha$-minimizer of $\mc E^{\alpha}_h$. Then there exist rotations $R_h\colon S\to SO(3)$, constant rotations $\overline R _h\in SO(3)$ and constant vectors $c_h\in \R^3$ such that the quantities defined in \eqref{eq:defuv} satisfy, up to (non relabelled) subsequences, the convergences \eqref{eq:convuv} and \eqref{eq:Gv}, and the limit pair $(u,v)$ minimizes the functional $\mc E_{\rm ben, lin}\colon W^{1,2}(S;\R^2)\times W^{2,2}(S)\to [0,+\infty]$, defined as
		\begin{equation}\label{eq:linbendingenergy}
			\mc E_{\rm ben, lin}(u,v):=\begin{cases}
				\displaystyle\frac{1}{24}\int_SQ_2((\nabla')^2v(x'))\d x', &\text{ if \eqref{eq:linconstr} holds},\\
				+\infty, & \text{ otherwise}.
			\end{cases}
		\end{equation}
		
		{\rm\bf [Von K\'arm\'an theory ($\boldsymbol{\alpha=2}$) \cite{friesecke.james.mueller2}]}
		Assume that $W$ satisfies \eqref{hyp:bending} and let $y_h$ be a $2$-minimizer of $\mc E^{2}_h$. Then there exist rotations $R_h\colon S\to SO(3)$, constant rotations $\overline R _h\in SO(3)$ and constant vectors $c_h\in \R^3$ such that the quantities defined in \eqref{eq:defuv} satisfy, up to (non relabelled) subsequences, the convergences \eqref{eq:convuv}, \eqref{eq:Gv} and \eqref{eq:G0vK}, and the limit pair $(u,v)$ minimizes the von K\'arm\'an energy $\mc E_{\rm vK}\colon W^{1,2}(S;\R^2)\times W^{2,2}(S)\to [0,+\infty)$, defined as
		\begin{equation}\label{eq:vKenergy}
			\mc E_{\rm vK}(u,v):=\frac 12 \int_S Q_2\left(e'(u(x'))+\frac 12 \nabla' v(x')\otimes \nabla' v(x')\right) \d x'+\frac{1}{24}\int_SQ_2((\nabla')^2v(x'))\d x'.
		\end{equation}
		
		{\rm\bf [Linearized von K\'arm\'an theory ($\boldsymbol{\alpha>2}$) \cite{friesecke.james.mueller2}]}
		Assume that $W$ satisfies \eqref{hyp:bending} and let $y_h$ be an $\alpha$-minimizer of $\mc E^{\alpha}_h$. Then there exist rotations $R_h\colon S\to SO(3)$, constant rotations $\overline R _h\in SO(3)$ and constant vectors $c_h\in \R^3$ such that the quantities defined in \eqref{eq:defuv} satisfy, up to (non relabelled) subsequences, the convergences \eqref{eq:convuv}, \eqref{eq:Gv} and \eqref{eq:G0vKlin}, and the limit pair $(u,v)$ minimizes the functional $\mc E_{\rm vK, lin}\colon W^{1,2}(S;\R^2)\times W^{2,2}(S)\to [0,+\infty)$, defined as
		\begin{equation}\label{eq:vKlinenergy}
			\mc E_{\rm vK, lin}(u,v):=\frac 12 \int_S Q_2\left(e'(u(x'))\right) \d x'+\frac{1}{24}\int_SQ_2((\nabla')^2v(x'))\d x'.
		\end{equation}
	\end{prop}

	\subsection{Pressure effects}
	
	We can now state our results regarding the asymptotic behaviour  as $h\to 0$ of $\overline{\mc E}_h^{\alpha,\pi}$ (of ${\mc E}_h^{\alpha,\pi}$ if $\alpha=0$), namely when the action of the pressure term is taken into account. We will see that for high rescalings, namely in the von K\'arm\'an-type regimes $\alpha\ge 2$, pressure loads interact with the elastic energy density giving rise to nontrivial effects, while for middle rescalings, namely in the bending regimes $\alpha\in[1,2)$, the presence of pressure plays no role in terms of minimizers. In the lowest rescaling $\alpha=0$, corresponding to the membrane regime, the espression \eqref{eq:membranepien} below suggests that pressure effects may still be persistent in the limit. Although we strongly believe this is not case (namely, pressure effects survive in the limit just for high rescaling) motivated by the previous considerations, we are not able to prove that minimizers of \eqref{eq:membranepien} do not depend on $\pi$. We postpone a discussion on this topic to Section~\ref{sec:membrane} where in particular we propose conjecture \eqref{conjecture}, and we present a supporting example.
	
	It will be useful to introduce the following function, defined on $2\times 2$ matrices $G$:
	\begin{equation}\label{def:Qpi}
		Q_2^\pi(G):=\min\limits_{a\in \R^3}\{Q_3(G+a\otimes e_3)+2\pi a_3\}.
	\end{equation}
	Since $Q_3$ is a quadratic form, one can easily show that
	\begin{equation}\label{eq:Qrewriting}
		Q_2^\pi(G)= Q_2(G)+\pi\mc L G+\pi^2\kappa,
	\end{equation}
	for some linear operator $\mc L\in Lin(\R^{2\times 2}_{\sym};\R)$ and some constant $\kappa\in \R$ (which both depend on $W$), where $Q_2$ has been defined in \eqref{def:Q2}.
	
	The next theorem states our main result.
	\begin{thm}\label{thm:main}
		{\rm\bf [Membrane theory ($\boldsymbol{\alpha=0}$)]} Assume that $W$ is continuous and it satisfies \eqref{hyp:membrane} with $p\ge 3$, and if $p=3$ assume $\pi>-c_1$, where $c_1$ is the constant appearing in \eqref{hyp:Wgrowth}. Then $|\inf \mc E^{0,\pi}_h|\le C$. Let now $y_h$ be a $0$-minimizer of $\mc E^{0,\pi}_h$. Then there exist constant vectors $c_h\in \R^3$ such that, up to (non relabelled) subsequences, $y_h-c_h$ converges as $h\to 0$ in the weak topology of $W^{1,p}(\Omega;\R^3)$ to a minimizer of the functional $\mc E^\pi_{\rm mem}\colon L^p(\Omega;\R^3)\to (-\infty,+\infty]$, defined as
		\begin{equation}\label{eq:membranepien}
			\mc E^\pi_{\rm mem}(y):=\begin{cases}
				\displaystyle\int_S\mc Q(W^\pi)_0(\nabla' y(x'))\d x', &\text{ if } y\in W^{1,p}(\Omega;\R^3) \text{ and }\partial_3y=0\\
				+\infty, & \text{ otherwise,}
			\end{cases}
		\end{equation}
		where $W^\pi(F):=W(F)+\pi\det F$. Moreover, $\partial_3 y_h$ vanishes as $h\to 0$ in the strong topology of $L^p(\Omega;\R^3)$.
		
		{\rm\bf [Nonlinear bending theory ($\boldsymbol{\alpha=1}$)]} 	Assume that $W$ satisfies \eqref{hyp:bending} together with either \eqref{eq:orientationpreserving} or \eqref{hyp:pineg}, and assume that the planar set $S$ satisfies:
		\begin{equation}\label{eq:regS}
			\begin{gathered}
				\text{there exists a closed subset $N\subseteq \partial S$ with null $\mc H^1$-measure with the property that}\\
				\text{the outer unit normal to $S$ exists and it is continuous on $\partial S\setminus N$.}
			\end{gathered}
		\end{equation}
		Then $-C h^{2}\le \inf \overline{\mc E}^{1,\pi}_h\le 0$. Let now $y_h$ be a $1$-minimizer of $\overline{\mc E}^{1,\pi}_h$. Then  there exist constant vectors $c_h\in \R^3$ such that, up to (non relabelled) subsequences, $y_h-c_h$ converges as $h\to 0$ in the strong topology of $W^{1,2}(\Omega;\R^3)$ to a minimizer of the bending energy $\mc E_{\rm ben}$, defined in \eqref{eq:bendingenergy}. 	Moreover, \eqref{eq:gradconv}, \eqref{eq:Ghconv} and \eqref{eq:Gdecomp} hold true.

		{\rm\bf [Linearized bending theory ($\boldsymbol{\alpha\in(1,2)}$)]}
		Assume that $W$ satisfies \eqref{hyp:bending} together with either \eqref{eq:orientationpreserving} or \eqref{hyp:pineg}. Let $S$ be simply connected if $\alpha\in (1,3/2)$. Then $-C h^{2\alpha}\le \inf \overline{\mc E}^{\alpha,\pi}_h\le 0$. Let now $y_h$ be an $\alpha$-minimizer of $\overline{\mc E}^{\alpha,\pi}_h$. Then there exist rotations $R_h\colon S\to SO(3)$, constant rotations $\overline R _h\in SO(3)$ and constant vectors $c_h\in \R^3$ such that the quantities defined in \eqref{eq:defuv} satisfy, up to (non relabelled) subsequences, the convergences \eqref{eq:convuv} and \eqref{eq:Gv}, and the limit pair $(u,v)$ minimizes the functional $\mc E_{\rm ben, lin}$ defined in \eqref{eq:linbendingenergy}.

		{\rm\bf [Von K\'arm\'an theory ($\boldsymbol{\alpha=2}$)]}
		Assume that $W$ satisfies \eqref{hyp:bending} together with either \eqref{eq:orientationpreserving} or \eqref{hyp:pineg}. Then $-C h^{4}\le \inf \overline{\mc E}^{2,\pi}_h\le 0$. Let now $y_h$ be a $2$-minimizer of $\overline{\mc E}^{2,\pi}_h$. Then there exist rotations $R_h\colon S\to SO(3)$, constant rotations $\overline R _h\in SO(3)$ and constant vectors $c_h\in \R^3$ such that the quantities defined in \eqref{eq:defuv} satisfy, up to (non relabelled) subsequences, the convergences \eqref{eq:convuv}, \eqref{eq:Gv}, and \eqref{eq:G0vK}, and the limit pair $(u,v)$ minimizes the functional $\mc E^\pi_{\rm vK}\colon W^{1,2}(S;\R^2)\times W^{2,2}(S)\to \R$, defined as
		\begin{equation}\label{eq:vKenergypi}
			\mc E^\pi_{\rm vK}(u,v):= \mc E_{\rm vK}(u,v)+\pi\!\int_S\!\left(\frac 12 \mc L\left(e'(u(x')){+}\frac 12 \nabla' v(x')\otimes \nabla' v(x')\right){+}\div' u(x'){+}\frac 12|\nabla' v(x')|^2\right)\!\! \d x',
		\end{equation}
		where $\mc E_{\rm vK}$ and $\mc L$ have been introduced in \eqref{eq:vKenergy} and \eqref{eq:Qrewriting}, respectively.

		{\rm\bf [Linearized von K\'arm\'an theory ($\boldsymbol{\alpha>2}$) }
		Assume that $W$ satisfies \eqref{hyp:bending} together with either \eqref{eq:orientationpreserving} or \eqref{hyp:pineg}. Then $-C h^{2\alpha}\le \inf \overline{\mc E}^{\alpha,\pi}_h\le 0$. Let now $y_h$ be an $\alpha$-minimizer of $\overline{\mc E}^{\alpha,\pi}_h$. Then there exist rotations $R_h\colon S\to SO(3)$, constant rotations $\overline R _h\in SO(3)$ and constant vectors $c_h\in \R^3$ such that the quantities defined in \eqref{eq:defuv} satisfy, up to (non relabelled) subsequences, the convergences \eqref{eq:convuv}, \eqref{eq:Gv} and \eqref{eq:G0vKlin}, and the limit pair $(u,v)$ minimizes the functional $\mc E^\pi_{\rm vK, lin}\colon W^{1,2}(S;\R^2)\times W^{2,2}(S)\to \R$, defined as
		\begin{equation}\label{eq:vKlinenergypi}
			\mc E^\pi_{\rm vK, lin}(u,v):=\mc E_{\rm vK, lin}(u,v)+\pi\int_S\left(\frac 12 \mc L(e'(u(x'))+\div' u(x')\right) \d x',
		\end{equation}
		where $\mc E_{\rm vK, lin}$ and $\mc L$ have been introduced in \eqref{eq:vKlinenergy} and \eqref{eq:Qrewriting}, respectively.
	\end{thm}
	
	The regularity request \eqref{eq:regS} in the regime $\alpha=1$, not needed in the pressureless case, will be used in Theorem~\ref{thm:Gammabending} for the construction of a recovery sequence in the $\Gamma$-convergence analysis, exploiting the following approximation result, whose proof can be found in \cite{Hornung}.
	\begin{thm}\label{thm:density}
		Assume that $S\subseteq \R^2$ is a bounded Lipschitz domain which satisfies \eqref{eq:regS}. Then $W^{2,2}_{\rm iso}(S;\R^3)\cap C^{\infty}(\overline S;\R^3)$ is dense in $W^{2,2}_{\rm iso}(S;\R^3)$ with respect to its strong topology.
	\end{thm}
	
	\subsection{Example of interaction in the von K\'arm\'an regimes}
	
	In the case of isotropic materials, where
	\begin{equation*}
		Q_3(F)=2\mu|\sym F|^2+\lambda (\tr F)^2,
	\end{equation*}
	for some constants $\mu>0$ and $\lambda>-\frac 23 \mu$ called Lam\'e coefficients, one can explicitely compute $Q_2^\pi$ directly from the definition \eqref{def:Qpi}. As a result one obtains the expression
	\begin{equation*}
		Q_2^\pi(G)=\underbrace{2\mu|\sym G|^2+\frac{2\mu\lambda}{2\mu+\lambda} (\tr G)^2}_{=Q_2(G)}-\pi\underbrace{\frac{2\lambda}{2\mu+\lambda}\tr G}_{=-\mc L G}-\pi^2\underbrace{\frac{1}{2\mu+\lambda}}_{=-\kappa}.
	\end{equation*}
	
	With this specific choice, we can thus better characterize \eqref{eq:vKenergypi} and \eqref{eq:vKlinenergypi}. Indeed, by the previous computation we deduce
	\begin{subequations}
		\begin{equation}\label{eq:vKexplicit}
			\mc E^\pi_{\rm vK}(u,v):= \mc E_{\rm vK}(u,v)+\frac{2\mu\pi}{2\mu+\lambda}\int_S\left(\div' u(x')+\frac 12|\nabla' v(x')|^2\right)\d x',
		\end{equation}
		and 
		\begin{equation}\label{eq:vKlinexplicit}
			\mc E^\pi_{\rm vK, lin}(u,v):=\mc E_{\rm vK, lin}(u,v)+\frac{2\mu\pi}{2\mu+\lambda}\int_S\div' u(x') \d x'.
		\end{equation}    
	\end{subequations}
	
	In particular, the Euler-Lagrange equations of \eqref{eq:vKexplicit} (and similarly the ones of \eqref{eq:vKlinexplicit}) present the following two differences with respect to the pressureless case: the equation for $v$ features the additional term $-\frac{2\mu\pi}{2\mu+\lambda}\Delta'v$, while in the equation for $u$ the non-homogeneous Neumann boundary condition $-\frac{2\mu\pi}{2\mu+\lambda}n_{\partial S}$, where $n_{\partial S}$ denotes the outward unit normal to $S$, replaces the homogeneous one.
	
	Thus, on the one hand pressure interacts with the bulk energy perturbing the vertical displacement of the thin material, on the other hand the bulk energy interacts with pressure resulting in a different horizontal pressure-like term acting on its boundary $\partial S$.
	
	\section{The membrane regime}\label{sec:membrane}
	
	We first focus on the proof of Theorem \ref{thm:main} in the case $\alpha=0$. Unlike the higher regimes, in this setting it will be a simple byproduct of Proposition \ref{prop:known}.
	\begin{proof}[Proof of Theorem \ref{thm:main} ($\alpha=0$)]
		Observing that we can write $\mc E^{0,\pi}_h(y)=\int_\Omega W^\pi(\nabla_hy(x)) \d x$ for $y\in \dom \mc E^{0,\pi}_h$, we just need to show that $W^\pi$ satisfies the assumptions \eqref{hyp:membrane} as well, in order to directly apply Proposition \ref{prop:known} ($\alpha=0$). 
		
		Clearly, $W^\pi$ is continuous and \eqref{hyp:orientation} is fulfilled. By recalling that $|\det F|\le |F|^3$, exploiting \eqref{hyp:orientation} we also infer
		\begin{equation*}
			W^\pi(F)\ge c_1|F|^p-c_2+\pi\det F\ge c_1|F|^p-c_2-\pi^-|F|^3,
		\end{equation*}
		where we employed the notation $\pi^-:=-(\pi\wedge 0)$. If $p=3$, the validity of \eqref{hyp:Wgrowth} now directly follows from the assumption $\pi>-c_1$, while for $p>3$ it comes from a straightforward application of the weighted Young's inequality.
		
		Finally, whenever $\det F\ge \delta$ we simply have ($|\det F|\le |F|^3\le 1+|F|^p$ since $p\ge 3$)
		\begin{equation*}
			W^\pi(F)\le C_\delta(1+|F|^p)+|\pi|\det F\le C_\delta\vee|\pi|(1+|F|^p),
		\end{equation*}
		thus also \eqref{hyp:pbounddelta} is fulfilled and we conclude.
	\end{proof}
	
	It would be interesting to understand how the structure of $W^\pi$ affects $(W^\pi)_0$ and more importantly $\mc Q (W^\pi)_0$, namely how they can be written in terms of $W_0$ and $\mc Q W_0$. Although the dependence of $(W^\pi)_0$ on $\pi$ may be quite nasty (see \eqref{eq:Wpi0ex} and \eqref{eq:gipi}), in view of the previous considerations about the higher rescalings we strongly believe and we conjecture that after the quasiconvexification procedure one actually has
	\begin{equation}\label{conjecture}
		\mc Q (W^\pi)_0(\xi)=\mc Q W_0(\xi) +c_\pi,\quad\text{for all }\xi\in \R^{3\times 2},
	\end{equation}
	where $c_\pi\in\R$ is a certain constant. We are not able to show the validity of \eqref{conjecture}, but we now provide an example which suggests, at least formally and at least for positive pressure, that it may be true.

	\subsection{Formal computation of $\mc Q(W^\pi)_0$ in a specific case}\label{sec:examplemembrane}
	We consider the density
	\begin{equation*}
		W(F):=\begin{cases}\displaystyle
			\frac{|F_1|^3+|F_2|^3+|F_3|^3}{3}+\frac{1}{\det F},&\text{if }\det F>0,\\
			+\infty,&\text{otherwise,}
		\end{cases}
	\end{equation*}
	which clearly satisfies assumptions \eqref{hyp:membrane} with $p=3$.
	
	Observing that
	\begin{equation*}
		(W^\pi)_0(\xi)=\inf\limits_{a\in\R^3}(W(\xi|a)+\pi \det(\xi|a))=\inf\limits_{\substack{{a\in\R^3,}\\{(\xi_1\wedge \xi_2)\cdot a >0}}}(W(\xi|a)+\pi (\xi_1\wedge \xi_2)\cdot a),
	\end{equation*}
	in this example we have
	\begin{equation*}
		(W^\pi)_0(\xi)=\frac{|\xi_1|^3+|\xi_2|^3}{3}+ \inf\limits_{\substack{{a\in\R^3,}\\{(\xi_1\wedge \xi_2)\cdot a >0}}}\left(\frac{|a|^3}{3}+\frac{1}{(\xi_1\wedge \xi_2)\cdot a}+\pi (\xi_1\wedge \xi_2)\cdot a\right).
	\end{equation*}
	The minimum point $\bar a$ can be explicitely computed by solving the system
	\begin{equation*}
		\begin{cases}
			\displaystyle |\bar a|\bar a=\left(\frac{1}{((\xi_1\wedge \xi_2)\cdot \bar a)^2}-\pi\right)\xi_1\wedge \xi_2,\\
			(\xi_1\wedge \xi_2)\cdot \bar a>0,
		\end{cases}
		\iff
		\begin{cases}
			\displaystyle\bar a =|\bar a|\frac{\xi_1\wedge \xi_2}{|\xi_1\wedge \xi_2|},\\
			|\xi_1\wedge \xi_2||\bar a|^4+\pi |\xi_1\wedge \xi_2|^2|\bar a|^2-1=0,
		\end{cases}
	\end{equation*}
	which has the unique solution
	\begin{equation*}
		\bar a= \sqrt{\frac{\sqrt{4|\xi_1\wedge \xi_2|+\pi^2|\xi_1\wedge \xi_2|^4}-\pi|\xi_1\wedge \xi_2|^2}{2|\xi_1\wedge \xi_2|}}\frac{\xi_1\wedge \xi_2}{|\xi_1\wedge \xi_2|}.
	\end{equation*}
	By some simple computations we thus obtain
	\begin{equation}\label{eq:Wpi0ex}
		(W^\pi)_0(\xi)=
		\frac{|\xi_1|^3+|\xi_2|^3}{3}+\frac{\sqrt{2}}{3}g_\pi(\xi_1\wedge\xi_2),
	\end{equation}
	where
	\begin{equation}\label{eq:gipi}
		g_\pi(v)=\begin{cases}\displaystyle
			\frac{4}{\sqrt{|v|\sqrt{4|v|+\pi^2|v|^4}-\pi |v|^3}}+\pi \sqrt{|v|\sqrt{4|v|+\pi^2|v|^4}-\pi |v|^3},&\text{if }|v|>0,\\
			+\infty,&\text{otherwise.}
		\end{cases}
	\end{equation}
	
	We now focus on the case of nonnegative pressure $\pi\ge0$, and we limit ourselves to some formal considerations. Since the map $\xi\mapsto \frac{|\xi_1|^3+|\xi_2|^3}{3}$ is convex and in view of \cite[Theorem 6.26]{Dacorogna} we expect that
	\begin{equation}\label{eq:boh}
		\mc Q (W^\pi)_0(\xi)=\frac{|\xi_1|^3+|\xi_2|^3}{3}+\frac{\sqrt{2}}{3}(\mc Cg_\pi)(\xi_1\wedge\xi_2),
	\end{equation}
	where $\mc C$ denotes the convex envelope. At this level, the previous passage is just formal for two reasons: first there is a sum involved in the quasiconvexification procedure, and then $g_\pi$ takes the value $+\infty$ (not allowed in \cite[Theorem 6.26]{Dacorogna}).
	
	Since $g_\pi$ is radial, namely $g_\pi(v)=\rho_\pi(|v|)$ for some even function $\rho_\pi\colon\R\to (-\infty,+\infty]$, it is easy to see that
	\begin{equation*}
		(\mc Cg_\pi)(v)=(\mc C\rho_\pi)(|v|),
	\end{equation*}
	so we just need to compute $\mc C\rho_\pi$. To this aim, we observe that $\rho_\pi(x)=h(j(x))$ for $x\ge 0$, where
	\begin{equation*}
		j(x)=x\sqrt{4x+\pi^2 x^4}-\pi x^3,\qquad\text{and}\qquad h(y)=\begin{cases}
			\frac{4}{\sqrt{y}}+\pi\sqrt{y},&\text{if }y>0,\\
			+\infty,&\text{otherwise}.
		\end{cases}
	\end{equation*}
	If $\pi>0$, it is immediate to check that $j$ is strictly increasing, $j(0)=0$ and $\lim\limits_{x\to +\infty}j(x)=\frac 2\pi$, so that $0\le j(x)<\frac 2\pi$ for all $x\ge 0$. Since $h'(y)=\frac{\pi y-4}{2y\sqrt{y}}$, we thus deduce that $h'(j(x))<0$ for all $x>0$, whence we infer that $\rho_\pi$ is strictly decreasing in $(0,+\infty)$. In particular, we can compute $\lim\limits_{x\to +\infty}\rho_\pi(x)=h(\frac 2\pi)=3\sqrt{2\pi}$. If $\pi=0$ instead, it simply holds $\rho_0(x)=\frac{2\sqrt{2}}{\sqrt{x\sqrt{x}}}$, which is strictly decreasing and vanishes at infinity.
	
	In both situations, previous argument yields $\mc C\rho_\pi(x)\equiv 3\sqrt{2\pi}$, and so by \eqref{eq:boh} we have
	\begin{equation*}
		\mc Q(W^\pi)_0(\xi)=\frac{|\xi_1|^3+|\xi_2|^3}{3}+2\sqrt{\pi}= \mc Q W_0(\xi)+2\sqrt{\pi},
	\end{equation*}
	namely conjecture \eqref{conjecture} holds true in this specific case (up to some formal computations).
	
	\section{$\Gamma$-convergence analysis for higher scalings}\label{sec:highscalings}
	
	The rest of the paper is concerned with the proof of Theorem \ref{thm:main} in case $\alpha\ge 1$. Here and henceforth we thus tacitly assume \eqref{hyp:bending}. The stated results will be a standard consequence of a $\Gamma$-convergence analysis \cite{braides,dalmaso}. We will first show  that sequences with bounded rescaled energy are compact with respect to the desired topologies, and then we will compute the $\Gamma$-limit as $h\to 0$ of $\frac{1}{h^{2\alpha}}\overline{\mc E}_h^{\alpha,\pi}$ in the different regimes, obtaining the stated expressions.
	
	\begin{prop}[Compactness]\label{prop:bounds}
		Let $y_h$ be such that $\overline{\mc E}_h^{\alpha,\pi}(y_h)\le C h^{2\alpha}$. If \eqref{eq:orientationpreserving} is in force, then there holds
		\begin{equation}\label{eq:boundsdet}
			\int_\Omega W(\nabla_h y_h)\d x+\int_{\{|\det\nabla_h y_h-1|\le 1\}}|\det\nabla_h y_h-1|^2 \d x\le C h^{2\alpha}.
		\end{equation}
		In particular, all the compactness results listed in Proposition~\ref{prop:compactness} hold true.
		
		If instead we assume \eqref{hyp:pineg}, in addition to \eqref{eq:boundsdet} there also holds
		\begin{equation}\label{eq:boundsdetbig}
			\int_{\{|\det\nabla_h y_h-1|> 1\}}|\det\nabla_h y_h-1| \d x\le C h^{2\alpha}.
		\end{equation}
		Moreover, in both cases one has
		\begin{equation}\label{eq:infbounds}
			-C h^{2\alpha}\le \inf \overline{\mc E}^{\alpha,\pi}_h\le 0.
		\end{equation}
	\end{prop}
	\begin{proof}
		We will prove the result assuming \eqref{eq:orientationpreserving}, highlighting the changes needed when \eqref{hyp:pineg} holds true instead.
		
		For the sake of clarity let us define the set
		\begin{equation}\label{eq:setdet}
			\Omega_h^-:=\{x\in \Omega:\, |\det\nabla_h y_h(x)-1|\le 1\}.
		\end{equation}
		By the expression \eqref{eq:energybar}, we now deduce
		\begin{align}\label{eq:pi1}
			\int_\Omega W(\nabla_h y_h)\d x&\le C h^{2\alpha}+h^\alpha\pi\int_\Omega\Big(1-\det\nabla_h y_h\Big)\d x\le Ch^{2\alpha}+ h^\alpha\pi\int_{\Omega_h^-}|\det\nabla_h y_h-1|\d x\nonumber\\
			&\le Ch^{2\alpha}+ Ch^\alpha \|\det\nabla_h y_h-1\|_{L^2(\Omega_h^-)}.
		\end{align}
		In the second inequality above we exploited \eqref{eq:orientationpreserving}, so that $\pi>0$ and $\det\nabla_hy_h>0$ a.e. in $\Omega$ since the energy is finite. By means of \eqref{eq:growthdet} we thus infer
		\begin{equation}\label{eq:goal}
			\|\det\nabla_h y_h-1\|_{L^2(\Omega_h^-)}^2\le Ch^{2\alpha}+ Ch^\alpha \|\det\nabla_h y_h-1\|_{L^2(\Omega_h^-)},
		\end{equation}
		whence 
		\begin{equation}\label{eq:goal2}
			\|\det\nabla_h y_h-1\|_{L^2(\Omega_h^-)}^2\le Ch^{2\alpha},
		\end{equation}
		and so we conclude the proof of \eqref{eq:boundsdet}.
		
		On the other hand, if \eqref{hyp:pineg} is in force, we can use the inequality
		\begin{equation*}
			\int_\Omega W(\nabla_h y_h)\d x\le Ch^{2\alpha}+ Ch^\alpha \|\det\nabla_h y_h-1\|_{L^2(\Omega_h^-)}+Ch^{\alpha}\|\det\nabla_h y_h-1\|_{L^1(\Omega\setminus\Omega_h^-)},
		\end{equation*}
		in place of the stronger \eqref{eq:pi1}. Up to absorbing the last term above to the left-hand side by means of \eqref{hyp:pineg}, we again end up with \eqref{eq:goal}. Thus \eqref{eq:goal2} follows, and therefore \eqref{eq:boundsdetbig} as well from \eqref{hyp:pineg}. As a consequence, \eqref{eq:boundsdet} is proved.
		
		We now notice that the upper bound in \eqref{eq:infbounds} is trivial, since $\overline{\mc E}^{\alpha,\pi}_h({id}_h)=0$. To prove also the lower bound, let $y_h$ be such that $\overline{\mc E}^{\alpha,\pi}_h(y_h)\le \inf \overline{\mc E}^{\alpha,\pi}_h + h^{2\alpha}$, so that \eqref{eq:boundsdet} holds (also \eqref{eq:boundsdetbig} if \eqref{hyp:pineg} is assumed). If \eqref{eq:orientationpreserving} holds we can now estimate
		\begin{align*}
			\inf \overline{\mc E}^{\alpha,\pi}_h&\ge  \overline{\mc E}^{\alpha,\pi}_h(y_h)-h^{2\alpha}\ge \pi h^\alpha \int_{\Omega_h^-}(\det\nabla_h y_h -1)\d x-h^{2\alpha}\\
			&\ge -C h^\alpha \|\det\nabla_h y_h-1\|_{L^2(\Omega_h^-)} -h^{2\alpha}\ge - C h^{2\alpha}, 
		\end{align*}
		and a similar argument works assuming \eqref{hyp:pineg}, so we conclude.
	\end{proof}

	\subsection{Kirchhoff regime}
	In this section we compute the $\Gamma$-limit as $h\to 0$ of the functional $\frac{1}{h^2}\overline{\mc E}^{1,\pi}_h$ in case $\alpha=1$, corresponding to nonlinear bending theory (due to Kirchoff), in the spirit of \cite{friesecke.james.mueller1}.
	\begin{thm}\label{thm:Gammabending}
		{\rm\bf [Nonlinear bending theory ($\boldsymbol{\alpha=1}$)]} Under the assumptions of Theorem~\ref{thm:main}, the functionals $\frac{1}{h^2}\overline{\mc E}^{1,\pi}_h$ Mosco converge as $h\to 0$ in the topology of $W^{1,2}(\Omega;\R^3)$ to the functional $\mc E_{\rm ben}+\frac 12 m_\pi |S|$, namely
		\begin{itemize}
			\item if $y_h\xrightharpoonup[h\to 0]{}y$ in the weak topology of $W^{1,2}(\Omega;\R^3)$, then one has 
			\begin{equation}\label{eq:Mliminf}
				\mc E_{\rm ben}(y)+\frac 12 m_\pi |S|\le \liminf\limits_{h\to 0}\frac{1}{h^2}\overline{\mc E}^{1,\pi}_h(y_h);
			\end{equation}
			\item for all $y\in W^{1,2}(\Omega;\R^3)$ there exists $y_h$ such that $y_h\xrightarrow[h\to 0]{}y$ in the strong topology of $W^{1,2}(\Omega;\R^3)$ for which
			\begin{equation}\label{eq:Mlimsup}
				\limsup\limits_{h\to 0}\frac{1}{h^2}\overline{\mc E}^{1,\pi}_h(y_h)\le  \mc E_{\rm ben}(y)+\frac 12 m_\pi |S|.
			\end{equation}			
		\end{itemize}
		The constant $m_\pi$ is defined as
		\begin{equation}\label{def:mpi}
			m_\pi:= \min\limits_{G\in \R^{2\times 2}_{\rm sym}}\{Q_2^\pi(G)+2\pi \tr G\},
		\end{equation}
		where $Q_2^\pi$ has been introduced in \eqref{def:Qpi}.
	\end{thm}
	\begin{proof}
		We perform the proof under the assumption \eqref{eq:orientationpreserving}. The case with \eqref{hyp:pineg} can be handled similarly by exploiting also \eqref{eq:boundsdetbig}.
		
		We start proving the liminf inequality \eqref{eq:Mliminf}. Without loss of generality we can assume that $\liminf\limits_{h\to 0}\frac{1}{h^2}\overline{\mc E}^{1,\pi}_h(y_h)$ is finite, so that by Proposition~\ref{prop:bounds} and then  Proposition~\ref{prop:compactness} we actually have that $y_h\xrightarrow[h\to 0]{}y$ strongly in $W^{1,2}(\Omega;\R^3)$, and $y\in W^{2,2}_{\rm iso}(S;\R^3)$.
		
		Moreover, by \eqref{eq:Gh}, we know that for $x\in \widehat\Omega_h:= S_h\times\left(-\frac 12,\frac 12\right)$ we can write $\nabla_h y_h(x)=R_h(x')(I+h G_h(x))$. Hence, by recalling the notation \eqref{eq:setdet}, we infer
		\begin{align}\label{eq:ineq1}
			\frac{1}{h^2}\overline{\mc E}^{1,\pi}_h(y_h)&=\frac{1}{h^2}\int_\Omega W(\nabla_h y_h)\d x+\frac{\pi}{h}\int_\Omega(\det\nabla_h y_h-1)\d x\nonumber\\
			&\ge \frac{1}{h^2}\int_{\widehat\Omega_h} W(I+h G_h)\d x+\frac{\pi}{h}\int_{\Omega_h^-}(\det\nabla_h y_h-1)\d x\\
			&=\frac{1}{h^2}\int_{\widehat\Omega_h}\!\! W(I+h G_h)\d x+\frac{\pi}{h}\int_{\Omega_h^-}(\det(I+h G_h){-}1)\d x+\frac{\pi}{h}\int_{\Omega_h^-\setminus \widehat\Omega_h}(\det\nabla_h y_h{-}1)\d x.\nonumber
		\end{align}
		We point out that in the second integral above we exploited the identity $G_h=0$ outside $\widehat\Omega_h$.
		
		By using the results contained in \cite[Proof of Theorem 6.1(i)]{friesecke.james.mueller1} and exploiting \eqref{eq:Ghconv}, one can prove that
		\begin{equation}\label{eq:ineq2}
			\liminf\limits_{h\to 0}\frac{1}{h^2}\int_{\widehat\Omega_h} W(I+h G_h)\d x\ge \frac 12 \int_\Omega Q_3(G(x)) \d x.
		\end{equation}
		Moreover, we observe that by \eqref{eq:boundsdet} and \eqref{eq:setvanish} we have
		\begin{equation}\label{eq:vanish1}
			\left| \frac{1}{h}\int_{\Omega_h^-\setminus \widehat\Omega_h}(\det\nabla_h y_h{-}1)\d x\right|\le \frac 1h \|\det\nabla_h y_h{-}1\|_{L^2(\Omega_h^-)}|\Omega\setminus\widehat\Omega_h|^{1/2}\le C|\Omega\setminus\widehat\Omega_h|^{1/2}\xrightarrow[h\to 0]{}0.
		\end{equation}
		In order to deal with the second term in the last line of \eqref{eq:ineq1} we first introduce the set
		\begin{equation}\label{eq:Bh}
			B_h:=\{x\in \Omega: h^{1/2}|G_h(x)|\le 1\},
		\end{equation}
		and we notice that $\lim\limits_{h\to 0}|\Omega\setminus B_h|=0$ by Chebyshev's inequality. We now split the integral under consideration into
		\begin{align*}
			\frac{1}{h}\int_{\Omega_h^-}(\det(I+h G_h){-}1)\d x&= \frac{1}{h}\int_{\Omega_h^-\cap B_h}(\det(I+h G_h){-}1)\d x+\frac{1}{h}\int_{\Omega_h^-\setminus B_h}(\det(I+h G_h){-}1)\d x\\
			&= \frac{1}{h}\int_{\Omega_h^-\cap B_h}(\det(I+h G_h){-}1)\d x+\frac{1}{h}\int_{(\Omega_h^-\setminus B_h)\cap \widehat\Omega_h}(\det\nabla_h y_h{-}1)\d x,
		\end{align*}
		and we observe that the second term above vanishes as $h\to 0$ since by \eqref{eq:boundsdet} there holds
		\begin{equation}\label{eq:vanish2}
			\frac{1}{h}\int_{(\Omega_h^-\setminus B_h)\cap \widehat\Omega_h}|\det\nabla_h y_h{-}1|\d x\le \frac 1h \|\det\nabla_h y_h{-}1\|_{L^2(\Omega_h^-)}|\Omega\setminus B_h|^{1/2}\le C|\Omega\setminus B_h|^{1/2}\xrightarrow[h\to 0]{}0.
		\end{equation}
		
		As regards the first term, we argue as in \cite{MoraRiva}, and we exploit \eqref{eq:sviluppodet} to write
		\begin{align*}
			& \frac{1}{h}\int_{\Omega_h^-\cap B_h}(\det(I+h G_h){-}1)\d x\\
			=& \int_{\Omega_h^-\cap B_h}\tr G_h\d x+ h \int_{\Omega_h^-\cap B_h}\iota_2(G_h)\d x+ h^2 \int_{\Omega_h^-\cap B_h}\det G_h \d x.
		\end{align*}
		By using \eqref{eq:Ghconv} and exploiting the definition \eqref{eq:Bh} of $B_h$, we now deduce
		\begin{subequations}\label{eq:vanish3}
			\begin{align}
				\bullet\, & h \int_{\Omega_h^-\cap B_h}|\iota_2(G_h)|\d x\le C h\|G_h\|_2^2\le C h \xrightarrow[h\to 0]{}0,\\
				\bullet\, & h^2 \int_{\Omega_h^-\cap B_h}|\det G_h| \d x\le C h^2 \int_{\Omega_h^-\cap B_h}|G_h|^2 \cdot h^{-1/2}\d x\le C h^{3/2} \|G_h\|_2^2\le C h^{3/2} \xrightarrow[h\to 0]{}0,\\
				\bullet\, &\lim\limits_{h\to 0} \int_{\Omega_h^-\cap B_h}\tr G_h\d x=\int_\Omega \tr G\d x,
			\end{align}
		\end{subequations}
		where the last limit holds true since $|\Omega\setminus (\Omega_h^-\cap B_h)|$ vanishes as $h\to 0$. Indeed, we already know that $\lim\limits_{h\to 0}|\Omega\setminus  B_h|=0$; moreover let us show that
		\begin{equation}\label{eq:claim}
			B_h\cap \widehat{\Omega}_h\subseteq \Omega_h^-,
		\end{equation}
		which would imply $|\Omega\setminus  \Omega_h^-|\le |\Omega\setminus  (B_h\cap \widehat{\Omega}_h)|\le |\Omega\setminus  B_h|+ |\Omega\setminus  \widehat{\Omega}_h| \xrightarrow[h\to 0]{} 0$. Pick $x\in B_h\cap \widehat{\Omega}_h$, so that by \eqref{eq:sviluppodet} we have
		\begin{equation}\label{eq:contained}
			\begin{aligned}
				|\det\nabla_h y_h(x)-1|&=|\det(I+h G_h(x))-1|= h|\tr G_h(x)+h \iota_2(G_h(x))+h^2\det G_h(x)|\\
				&\le C h |G_h(x)|(1+h |G_h(x)|+ h^2 |G_h(x)|^2)\le C h^{1/2}(1+h^{1/2} +h)\xrightarrow[h\to 0]{}0,
			\end{aligned}
		\end{equation}
		whence $x\in \Omega_h^-$ (for $h$ small enough) and so \eqref{eq:claim} holds true.
		
		By putting together \eqref{eq:ineq1}, \eqref{eq:ineq2}, \eqref{eq:vanish1}, \eqref{eq:vanish2} and \eqref{eq:vanish3} we thus have proved that
		\begin{equation*}
			\liminf\limits_{h\to 0}\frac{1}{h^2}\overline{\mc E}^{1,\pi}_h(y_h)\ge \frac 12\int_\Omega Q_3(G(x))\d x+\pi \int_\Omega \tr G(x)\d x.
		\end{equation*}
		By using definitions \eqref{def:Qpi}, \eqref{def:mpi}, and recalling \eqref{eq:Gdecomp} and \eqref{eq:Qrewriting} we finally infer
		\begin{align}\label{eq:finalstep}
			\liminf\limits_{h\to 0}\frac{1}{h^2}\overline{\mc E}^{1,\pi}_h(y_h)&\ge \frac 12\int_\Omega \Big(Q_2^\pi(G_{2\times 2}(x))+2\pi \tr G_{2\times 2}(x)\Big)\d x\nonumber\\
			&= \frac{1}{24}\int_S Q_2(II(x'))\d x'+ \frac 12 \int_S\Big(Q_2^\pi(G_0(x'))+2\pi \tr G_0(x')\Big) \d x'\\
			&\ge \mc E_{\rm ben}(y)+ \frac 12 m_\pi |S|.\nonumber
		\end{align}
		
		Let us now show the limsup inequality \eqref{eq:Mlimsup}. Since $\mc E_{\rm ben}$ is continuous in $W^{2,2}_{\rm iso}(S;\R^3)$, by Theorem~\ref{thm:density} it is enough to show the inequality for $y\in W^{2,2}_{\rm iso}(S;\R^3)\cap  C^{\infty}(\overline S;\R^3)$.
		
		Let $\overline{G}\in\R^{2\times 2}_{\rm sym}$ be such that 
		\begin{equation}\label{eq:Gbar}
			m_\pi=Q_2^\pi (\overline G)+2\pi\tr\overline G,
		\end{equation}
		and for $c,d\in C^\infty (\overline S;\R^3)$ let us consider
		\begin{equation}\label{eq:recovery}
			y_h(x):=y(x')+h\big(\nabla' y(x')\overline G x'+x_3 b(x')\big)+h^2\left(x_3 c(x')+\frac{x_3^2}{2}d(x')\right),
		\end{equation}
		where $b$ is the normal vector defined in \eqref{eq:normalvector}. Notice that clearly $y_h$ converges to $y$ strongly in $W^{1,2}(\Omega;\R^3)$ as $h\to 0$. By simple computations we obtain
		\begin{equation*}
			\nabla_h y_h(x)=R(x')+h\Big[\big(\nabla'(\nabla' y(x')\overline G x')|c(x')\big)+x_3(\nabla'b(x')|d(x'))\Big]+h^2 x_3\left(\nabla' c(x'){+}\frac{x_3}{2}\nabla'd(x')|0\right),
		\end{equation*}
		where we set $R(x'):=(\nabla' y(x')|b(x'))\in SO(3)$. By defining
		\begin{align*}
			G_h(x):=&\frac{R(x')^T\nabla_h y_h(x)-I}{h}\\
			=& R(x')^T\Big[\!\big(\nabla'(\nabla' y(x')\overline G x')|c(x')\big){+}x_3(\nabla'b(x')|d(x'))\Big]{+}h x_3 R(x')^T\!\left(\!\nabla' c(x'){+}\frac{x_3}{2}\nabla'd(x')|0\right)\!,
		\end{align*}
		one easily deduces that, as $h\to 0$, the sequence $G_h$ strongly converges in $L^2(\Omega;\R^{3\times 3})$ to the matrix $G$ defined as
		\begin{equation}\label{eq:G1G2}
			G(x):=\underbrace{ R(x')^T\big(\nabla'(\nabla' y(x')\overline G x')|c(x')\big)}_{=: G_1(x')}+x_3\underbrace{ R(x')^T(\nabla'b(x')|d(x'))}_{=: G_2(x')}.
		\end{equation}
		
		We now observe that, since $y\in W^{2,2}_{\rm iso}(S;\R^3)\cap  C^{\infty}(\overline S;\R^3)$, there holds (see also \cite[proof of Theorem 2.6]{agostiniani.lucantonio.lucic})
		\begin{equation*}
			\nabla' y(x')^T\nabla'(\nabla' y(x')\overline G x')=\overline G,
		\end{equation*}
		whence we can write
		\begin{equation}\label{eq:formG}
			G_{2\times 2}(x)=\overline G+x_3II(x').
		\end{equation}
		By exploiting the smoothness of $y,c,d$ we now infer
		\begin{align*}
			W(\nabla_h y_h(x))&= W(I+h G_h(x))=\frac{h^2}{2}Q_3(G_1(x')+x_3 G_2(x'))+O(h^3)\\
			&=\frac{h^2}{2}\Big(Q_3(G_1(x'))+x_3^2 Q_3(G_2(x'))+2x_3 D^2W(I)[G_1(x')][G_2(x')]\Big)+O(h^3),
		\end{align*}
		and
		\begin{align*}
			\det\nabla_h y_h(x)-1&=\det(I+h G_h(x))-1= h\tr G_h(x)+h^2 \iota_2(G_h(x))+h^3\det G_h(x)\\
			&=h(\tr G_1(x')+x_3 \tr G_2(x'))+O(h^2),
		\end{align*}
		where the Landau notation $O(\cdot)$ is uniform in $x\in \Omega$.
		
		Thus we obtain
		\begin{align*}
			\frac{1}{h^2}\overline{\mc E}^{1,\pi}_h(y_h)&= \frac{1}{h^2}\int_\Omega W(\nabla_h y_h(x))\d x +\frac \pi h\int_\Omega(\det\nabla_h y_h(x)-1) \d x\\
			&= \frac{1}{24}\int_S Q_3(G_2(x'))\d x'+\frac 12\int_S\big(Q_3(G_1(x'))+2\pi\tr G_1(x')\big)\d x'+O(h).
		\end{align*}
		
		By using the definition \eqref{eq:G1G2} of $G_1$ and $G_2$, and recalling \eqref{eq:formG}, we finally deduce
		\begin{align}\label{eq:finalineq}
			&\,\lim\limits_{h\to 0} \frac{1}{h^2}\overline{\mc E}^{1,\pi}_h(y_h)\nonumber\\
			=&\frac{1}{24}\int_S Q_3\begin{pmatrix}
				\begin{array}{c|c}
					II(x') & \frac 12(\nabla'y(x')^Td(x')+\nabla' b(x')^T b(x'))\\
					\hline 0\,\, 0 & b(x')\cdot d(x')
				\end{array}
			\end{pmatrix}\d x'\\
			&+ \frac 12\int_S\left[ Q_3\begin{pmatrix}
				\begin{array}{c|c}
					\overline G & \frac 12(\nabla'y(x')^Tc(x')+( b(x')^T\nabla'(\nabla'y(x')\overline G x'))^T)\\
					\hline 0\,\, 0 & b(x')\cdot c(x')
				\end{array}
			\end{pmatrix} +2\pi b(x')\cdot c(x')\right] \d x'\nonumber\\
			&+ \pi \tr \overline G |S|\nonumber.
		\end{align}
		
		We now choose
		\begin{subequations}\label{eq:dcbar}
			\begin{align}
				\overline d(x')&:= R(x')\left[2\overline\alpha(x')-\begin{pmatrix}\nabla' b(x')^T b(x')\\\overline\alpha_3(x')
				\end{pmatrix}\right],\\
				\overline c(x')&:= R(x')\left[\overline\beta-\begin{pmatrix}{\nabla'(\nabla'y(x')\overline G x')^T b(x')}\\{\overline\beta_3}\end{pmatrix}\right],
			\end{align}
		\end{subequations}
		where $\overline\alpha$ and $\overline\beta$ satisfy
		\begin{align*}
			Q_2(II(x'))&=Q_3(II(x')+\overline\alpha (x')\otimes e_3),\\
			Q_2^\pi (\overline G)&= Q_3(\overline G+\overline\beta \otimes e_3)+2\pi\overline \beta.
		\end{align*}
		Since under our assumption  $II\in C^\infty(\overline S;\R^{2\times 2}_{\rm sym})$, we observe that $\overline\alpha\in C^\infty(\overline S;\R^{3})$, whence $\overline d$ and $ \overline c$ belong to $C^\infty(\overline S;\R^{3})$ as well.
		
		By plugging $\overline d, \overline c$ defined in \eqref{eq:dcbar} into \eqref{eq:recovery}, from \eqref{eq:finalineq} and recalling \eqref{eq:Gbar} we finally deduce
		\begin{align*}
			\lim\limits_{h\to 0} \frac{1}{h^2}\overline{\mc E}^{1,\pi}_h(y_h)=\frac{1}{24}\int_S Q_3(II(x'))\d x'+\frac 12(	Q_2^\pi (\overline G)+2\pi\tr\overline G)|S|=\mc E_{\rm ben}(y)+\frac 12 m_\pi |S|,
		\end{align*}
		and we conclude.
	\end{proof}
	
	\subsection{Linearized bending and von K\'arm\'an-type regimes}
	We finally focus on the cases $\alpha>1$, corresponding to linearized bending ($\alpha\in(1,2)$) and von K\'arm\'an-type ($\alpha\ge 2$) regimes.
	\begin{thm}{\rm\bf [Von K\'arm\'an-type theory ($\boldsymbol{\alpha>1}$)]}
		Under the assumptions of Theorem~\ref{thm:main}, the functionals $\frac{1}{h^{2\alpha}}\overline{\mc E}^{\alpha,\pi}_h$ $\Gamma$-converge as $h\to 0$ with respect to the convergence \eqref{eq:defuv} and \eqref{eq:convuv} to the functional $\mc E^{\alpha,\pi}:W^{1,2}(S;\R^2)\times W^{2,2}(S)\to (-\infty,+\infty]$ defined as
		\begin{equation*}
			\mc E^{\alpha,\pi}(u,v):=\begin{cases}\displaystyle
				\mc E_{\rm ben, lin}(u,v)+\frac 12 m_\pi|S|,& \text{if }\alpha\in(1,2),\\
				\displaystyle\mc E^\pi_{\rm vK}(u,v)+\frac{\pi^2}{2}\kappa|S|,& \text{if }\alpha=2,\\
				\displaystyle\mc E^\pi_{\rm vK, lin}(u,v)+\frac{\pi^2}{2}\kappa|S|,& \text{if }\alpha>2,
			\end{cases}
		\end{equation*}
		where the constants $m_\pi$ and $\kappa$ have been introduced in \eqref{def:mpi} and \eqref{eq:Qrewriting}, respectively.
		
		In other words, the following two conditions hold:
		
		\begin{itemize}
			\item if $y_h\xrightarrow[h\to 0]{}(u,v)$ in the sense that there exist rotations $R_h\colon S\to SO(3)$, constant rotations $\overline R _h\in SO(3)$ and constant vectors $c_h\in \R^3$ for which \eqref{eq:defuv} and \eqref{eq:convuv} hold, then one has 
			\begin{equation}\label{eq:Gliminf}
				\mc E^{\alpha,\pi}(u,v)\le \liminf\limits_{h\to 0}\frac{1}{h^{2\alpha}}\overline{\mc E}^{\alpha,\pi}_h(y_h);
			\end{equation}
			\item for all $(u,v)\in W^{1,2}(S;\R^2)\times W^{2,2}(S)$ there exists $y_h$ such that $y_h\xrightarrow[h\to 0]{}(u,v)$ in the above sense, except that \eqref{eq:Ghconv2} holds for $\sym G_h$ in place of $G_h$ if $\alpha\ge 2$, for which
			\begin{equation}\label{eq:Glimsup}
				\limsup\limits_{h\to 0}\frac{1}{h^{2\alpha}}\overline{\mc E}^{\alpha,\pi}_h(y_h)\le  \mc E^{\alpha,\pi}(u,v).
			\end{equation}			
		\end{itemize}		
	\end{thm}
	\begin{proof}
		Again, we assume \eqref{eq:orientationpreserving}. The case with \eqref{hyp:pineg} can be treated in an analogous way by using also \eqref{eq:boundsdetbig}.\medskip
		
		\underline{\textbf{Liminf inequality.}}	Let us first prove the liminf inequality \eqref{eq:Gliminf}. Without loss of generality we can assume that $\liminf\limits_{h\to 0}\frac{1}{h^{2\alpha}}\overline{\mc E}^{\alpha,\pi}_h(y_h)$ is finite, so that $\det\nabla_h y_h$ is positive almost everywhere in $\Omega$ and Propositions~\ref{prop:bounds} and \ref{prop:compactness} apply.
		
		Moreover, by \eqref{eq:Gh2}, we can write $\nabla_h y_h(x)=\overline R_h R_h(x')(I+h^\alpha G_h(x))$. Hence, by recalling the notation \eqref{eq:setdet}, we infer
		\begin{equation}\label{eq:comp1}
			\begin{aligned}
				\frac{1}{h^{2\alpha}}\overline{\mc E}^{\alpha,\pi}_h(y_h)&=\frac{1}{h^{2\alpha}}\int_\Omega W(I+h^\alpha G_h(x)) \d x+\frac{\pi}{h^\alpha}\int_\Omega (\det(1+h^\alpha G_h(x))-1) \d x\\
				&\ge \frac{1}{h^{2\alpha}}\int_\Omega W(I+h^\alpha G_h(x)) \d x+\frac{\pi}{h^\alpha}\int_{\Omega_h^-} (\det(1+h^\alpha G_h(x))-1) \d x.
			\end{aligned}
		\end{equation}
		
		Similarly to \eqref{eq:Bh}, we now introduce the set
		\begin{equation*}
			B_h^\alpha:=\{x\in \Omega: h^{\alpha/2}|G_h(x)|\le 1\},
		\end{equation*}
		and we observe again that $\lim\limits_{h\to 0}|\Omega\setminus B^\alpha_h|=0$ by Chebyshev's inequality.
		
		By using the results contained in \cite[Proof of Theorem 6.1(i)]{friesecke.james.mueller1}, exploiting \eqref{eq:Ghconv2} one obtains
		\begin{equation}\label{eq:comp2}
			\liminf\limits_{h\to 0}\frac{1}{h^{2\alpha}}\int_\Omega W(I+h^\alpha G_h(x)) \d x\ge \frac 12 \int_\Omega Q_3(G(x)) \d x.
		\end{equation}
		
		Arguing similarly to the proof of Theorem~\ref{thm:Gammabending}, using \eqref{eq:sviluppodet} and observing that $B_h^\alpha\subseteq \Omega_h^-$ (for $h$ small enough) by reasoning as in \eqref{eq:contained}, we first write
		\begin{align*}
			&\,\frac{1}{h^\alpha}\int_{\Omega_h^-} (\det(1+h^\alpha G_h(x))-1) \d x\\
			=& \int_{B^\alpha_h}\tr G_h(x) \d x+h^\alpha\int_{B^\alpha_h}\iota_2(G_h(x)) \d x+h^{2\alpha} \int_{B^\alpha_h}\det G_h(x) \d x\\
			&+ \frac{1}{h^\alpha}\int_{\Omega_h^-\setminus B_h^\alpha}(\det\nabla_h y_h(x)-1)\d x,
		\end{align*}
		whence we deduce
		\begin{equation}\label{eq:comp3}
			\lim\limits_{h\to 0} \frac{1}{h^\alpha}\int_{\Omega_h^-} (\det(1+h^\alpha G_h(x))-1) \d x= \int_\Omega \tr G(x)\d x.
		\end{equation}
		
		By \eqref{eq:comp1}, \eqref{eq:comp2}, \eqref{eq:comp3} and recalling \eqref{def:Qpi} we finally infer
		\begin{align*}
			\liminf\limits_{h\to 0}\frac{1}{h^{2\alpha}}\overline{\mc E}^{\alpha,\pi}_h(y_h)&\ge \frac 12 \int_\Omega \Big(Q_3(G(x))+2\pi\tr G(x)\Big)\d x\\
			&\ge \frac 12 \int_\Omega \Big(Q_2^\pi(G_{2\times 2}(x))+2\pi\tr G_{2\times 2}(x)\Big)\d x.
		\end{align*}
		If $\alpha\in (1,2)$, we conclude by arguing as in \eqref{eq:finalstep} using \eqref{eq:Gv} in place of \eqref{eq:Gdecomp}. If $\alpha\ge 2$ instead, we conclude by also exploiting \eqref{eq:G0vK} and \eqref{eq:G0vKlin}, and by recalling \eqref{eq:Qrewriting}.	\medskip
		
		\underline{\textbf{Limsup inequality.}} We now show the validity of the limsup inequality \eqref{eq:Glimsup}. 
		
		\textbf{(Case $\boldsymbol{\alpha\in(1,2)}$).} Assume first that $\alpha\in (1,2)$ and let $(u,v)\in  W^{1,2}(S;\R^2)\times W^{2,\infty}(S)$ such that \eqref{eq:linconstr} holds true; in particular, $u$ belongs to $W^{1,\infty}(S;\R^2)$ by \cite[Theorem~7 and Proposition~9]{friesecke.james.mueller2}. The case $v\in W^{2,2}(S)$ can be treated via approximation arguing as in \cite[Section~6.4]{friesecke.james.mueller2}.
		
		By \cite[Theorem~7]{friesecke.james.mueller2} we can build an isometry
		\begin{equation*}
			\overline y_h(x')=\begin{pmatrix}
				x'+h^{2(\alpha-1)}\overline u_h(x')\\
				h^{\alpha-1}v(x')
			\end{pmatrix}\in W^{2,\infty}_{\rm iso}(S;\R^3),
		\end{equation*}
		where $\overline u_h(x')\in  W^{2,\infty}(S;\R^2)$ satisfies
		\begin{equation}\label{eq:baruconv}
			\overline u_h \xrightharpoonup[h\to 0]{*}u,\qquad \text{ weakly$^*$ in }W^{2,\infty}(S;\R^2).
		\end{equation}
		In particular we can write
		\begin{equation}\label{eq:yhexpansion}
			\overline y_h(x')=\begin{pmatrix}
				x' \\ 0
			\end{pmatrix}+O(h^{\alpha-1}),
		\end{equation}
		where the Landau notation holds in the sense of $W^{2,\infty}(S;\R^2)$.
		
		By defining $b_h(x'):=\partial_1 \overline y_h(x')\wedge \partial_2 \overline y_h(x')\in W^{1,\infty}(S;\R^3)$ we also have
		\begin{subequations}
			\begin{align}
				&b_h(x')= e_3-h^{\alpha-1}\begin{pmatrix}
					\nabla' v(x')^T\\ 0
				\end{pmatrix}+ O(h^{2(\alpha-1}),&&\text{ in }W^{1,\infty}(S;\R^3),\label{eq:bhexp}\\
				&\nabla' b_h(x')=- h^{\alpha-1} \begin{pmatrix}
					(\nabla')^2 v(x')\\ 0
				\end{pmatrix}+ O(h^{2(\alpha-1}),&&\text{ in }L^{\infty}(S;\R^{3\times 2}).\label{eq:nablabhexp}
			\end{align}
		\end{subequations}
		
		Similarly to \eqref{eq:recovery} we now make the ansatz
		\begin{equation*}
			y_h(x):=\overline y_h(x')+h x_3 b_h(x')+h^\alpha \nabla' \overline y_h(x')\overline G x'+ h^{\alpha+1}\left(x_3 \overline c+\frac{x_3^2}{2}d_h(x')\right),
		\end{equation*}
		where the matrix $\overline G$ is as in \eqref{eq:Gbar}, the vector $\overline c\in \R^3$ fulfils (recall \eqref{def:Qpi})
		\begin{equation}\label{eq:cbardef}
			Q_2^\pi(\overline G)=Q_3(\overline G+ \overline c\otimes e_3)+2\pi \overline c_3,
		\end{equation}
		and $d_h\in W^{1,\infty}(S;\R^3)$ satisfies
		\begin{equation}\label{eq:dh}
			d_h\xrightarrow[h\to 0]{} \overline d, \quad \text{in }L^2(S;\R^3),\qquad\text{ and }\qquad h^{1/3}d_h \xrightarrow[h\to 0]{} 0, \quad \text{in }W^{1,\infty}(S;\R^3),
		\end{equation}
		where $\overline d$ realizes (recall \eqref{def:Q2})
		\begin{equation}\label{eq:bard}
			Q_2((\nabla')^2 v(x'))=Q_3(-(\nabla')^2 v(x')+\overline d(x')\otimes e_3).
		\end{equation}
		
		Furthermore, we define
		\begin{equation*}
			R_h(x'):=(\nabla'\overline y_h(x')|b_h(x'))\in SO(3),\quad \overline R_h:=I,\quad\text{and } c_h:=0,
		\end{equation*}
		so that, after some simple computations, $\widetilde y_h\equiv y_h$ and
		\begin{align*}
			u_h(x')=&\frac{1}{h^{2(\alpha-1)}}\int_{-\frac 12}^{\frac 12}({y}_h'(x',x_3)-x')\d x_3=\overline u_h(x')+h^{2-\alpha}(\nabla'\overline y_h(x')\overline G x')'+\frac{h^{3-\alpha}}{24} d'_h(x'),\\
			v_h(x')=&\frac{1}{h^{\alpha-1}}\int_{-\frac 12}^{\frac 12}({y}_h)_3(x',x_3)\d x_3=v(x')+h(\nabla'\overline y_h(x')\overline G x')_3+\frac{h^2}{24}(d_h)_3(x'),\\
			\nabla_h y_h(x)=&R_h(x')+ hx_3(\nabla' b_h(x')|0)+  h^\alpha\big(\nabla'(\nabla' \overline y_h(x')\overline G x')|\overline c+x_3d_h(x')\big)+h^{\alpha+1} \frac{x_3^2}{2}(\nabla'd_h(x')|0),\\
			G_h(x)=& \frac{R_h(x')^T\nabla_h y_h(x')-I}{h^\alpha}\\
			=&R_h(x')^T\Big[ h^{1-\alpha}x_3(\nabla' b_h(x')|0)+\big(\nabla'(\nabla' \overline y_h(x')\overline G x')|\overline c+x_3 d_h(x')\big)+h \frac{x_3^2}{2}(\nabla'd_h(x')|0)\Big] .
		\end{align*}
		Observe that, by exploiting \eqref{eq:yhexpansion} and \eqref{eq:bhexp}, there hold 
		\begin{align*}
			&R_h(x')= I+ O(h^{\alpha-1}),&&\text{ in }W^{1,\infty}(S;\R^{3\times 3}),\\
			&\nabla'(\nabla' \overline y_h(x')\overline G x')= \nabla' \overline y_h(x')\overline G+ O(h^{\alpha-1})=\begin{pmatrix}
				\overline G\\ 0
			\end{pmatrix}+ O(h^{\alpha-1}),&&\text{ in }L^{\infty}(S;\R^{3\times 2}).
		\end{align*}
		By means of \eqref{eq:yhexpansion}, \eqref{eq:nablabhexp} and \eqref{eq:dh}, we thus infer
		\begin{align}
			u_h(x')=& \overline u_h(x')+O(h^{2-\alpha}) ,&&\text{ in }W^{1,\infty}(S;\R^{2}),\nonumber\\
			v_h(x')= &v(x')+ O(h) ,&&\text{ in }W^{1,\infty}(S),\nonumber\\
			\nabla_h y_h(x)=& I+ O(h^{\alpha-1}),&&\text{ in }L^{\infty}(\Omega;\R^{3\times 3}),\nonumber\\
			G_h(x)=& G_1+x_3 (G_2)_h(x')+O(h^{(\alpha-1)\wedge 2/3}),&&\text{ in }L^{\infty}(\Omega;\R^{3\times 3}),\label{eq:exprGh}
		\end{align}
		where
		\begin{align*}
			&G_1= \overline G+\overline c\otimes e_3,\\
			&(G_2)_h(x')= -(\nabla')^2 v(x')+d_h(x')\otimes e_3.
		\end{align*}
		
		Since $\alpha\in (1,2)$, by recalling \eqref{eq:baruconv} and \eqref{eq:dh} we hence deduce that the convergences in \eqref{eq:convuv} hold true with
		\begin{equation}\label{eq:G}
			G(x)=\underbrace{
				\overline G +\overline c\otimes e_3}_{=G_1}+ x_3\underbrace{\left(-(\nabla')^2 v(x')+\overline d(x')\otimes e_3\right)}_{=:G_2(x')}.
		\end{equation}
		By arguing as in \cite[Section 6.3]{friesecke.james.mueller2}, we first obtain
		\begin{equation}\label{eq:1}
			\begin{aligned}
				\lim\limits_{h\to 0}\frac{1}{h^{2\alpha}}\int_\Omega W(\nabla_h y_h(x))\d x&=\lim\limits_{h\to 0}\frac{1}{h^{2\alpha}}\int_\Omega W(I+h^\alpha G_h(x))\d x=\frac 12 \int_\Omega Q_3(G(x)) \d x\\&=\frac{1}{24} \int_S Q_3(G_2(x')) \d x'+\frac 12  Q_3(G_1)|S|.
			\end{aligned}
		\end{equation}
		By recalling \eqref{eq:sviluppodet}, we can then write
		\begin{equation*}
			\frac{1}{h^\alpha}\int_\Omega (\det\nabla_h y_h(x)-1) \d x
			= \int_\Omega \tr G_h(x) \d x+h^\alpha\int_\Omega \iota_2(G_h(x))\d x+h^{2\alpha}\int_\Omega \det G_h(x)\d x.     
		\end{equation*}
		By using \eqref{eq:exprGh} and \eqref{eq:dh}, we now have
		\begin{align*}
			&\int_\Omega \tr G_h(x) \d x=\tr G_1 |S|+O(h^{(\alpha-1)\wedge 2/3}),\\
			&h^\alpha\int_\Omega \iota_2(G_h(x))\d x= O(h^{\alpha-2/3}),\\
			&h^{2\alpha}\int_\Omega \det G_h(x)\d x= O(h^{2\alpha-1}),\\
		\end{align*}
		whence
		\begin{equation}\label{eq:2}
			\lim\limits_{h\to 0} \frac{\pi}{h^\alpha}\int_\Omega (\det\nabla_h y_h(x)-1) \d x=\pi\tr G_1 |S|.
		\end{equation}
		
		By putting together \eqref{eq:1} and \eqref{eq:2}, and recalling \eqref{eq:Gbar}, \eqref{eq:cbardef}, \eqref{eq:bard} and \eqref{eq:G} we finally obtain
		\begin{align*}
			\lim\limits_{h\to 0}\frac{1}{h^{2\alpha}}\overline{\mc E}_h^{\alpha,\pi}(y_h)&=\frac{1}{24} \int_S Q_3(G_2(x')) \d x'+\frac 12  (Q_3(G_1)+2\pi\tr G_1)|S|\\
			&=\frac{1}{24} \int_S Q_2((\nabla')^2v(x')) \d x'+\frac 12 \Big(Q_2^\pi(\overline G)+2\pi\tr\overline G\Big) |S|\\
			&= \mc E_{\rm ben, lin}(u,v)+\frac 12 m_\pi|S|=\mc E^{\alpha,\pi}(u,v),
		\end{align*}
		and we conclude if $\alpha\in (1,2)$.
		
		\textbf{(Case $\boldsymbol{\alpha\ge 2}$).} We now consider the case $\alpha\ge 2$. Fix $(u,v)\in  W^{1,2}(S;\R^2)\times W^{2,2}(S)$; actually, since both $\mc E^\pi_{\rm vK}$ and $\mc E^\pi_{\rm vK, lin}$ are continuous in $W^{1,2}(S;\R^2)\times W^{2,2}(S)$, we may assume without loss of generality that $u$ and $v$ are smooth in $\overline S$.
		
		We follow the ansatz of \cite{friesecke.james.mueller2}, see (119) and (124) therein, and we define
		\begin{equation*}
			y_h(x):={id}_h(x)+\begin{pmatrix}
				h^\alpha u(x')\\
				h^{\alpha-1}v(x')
			\end{pmatrix}- h^\alpha x_3 \begin{pmatrix}
				\nabla' v(x')^T\\
				0
			\end{pmatrix}+h^{\alpha+1}x_3 c(x')+h^{\alpha+1}\frac{x_3^2}{2} d(x'),
		\end{equation*}
		where $c$ and $d$ are smooth functions from $\overline S$ to $\R^3$ which will be chosen later. We also set
		\begin{equation*}
			R_h\equiv \overline R_h:=I,\qquad\text{ and }\qquad c_h=0,
		\end{equation*}
		so that, after some simple computation, one has $\widetilde y_h\equiv y_h$ and
		\begin{align}
			u_h(x')=&\frac{1}{h^\alpha}\int_{-\frac 12}^{\frac 12}({y}_h'(x',x_3)-x')\d x_3=u(x')+\frac{h}{24} d'(x'), \nonumber\\
			v_h(x')=&\frac{1}{h^{\alpha-1}}\int_{-\frac 12}^{\frac 12}({y}_h)_3(x',x_3)\d x_3=v(x')+\frac{h^2}{24}d_3(x'),\nonumber\\
			\nabla_h y_h(x)=&I+h^\alpha\begin{pmatrix}
				\begin{array}{c|c}
					\nabla' u(x')-x_3 (\nabla')^2 v(x') & -h^{-1}\nabla'v(x')^T\\
					\hline h^{-1}\nabla'v(x') & 0
				\end{array}
			\end{pmatrix}+h^\alpha \big(0\,|\,0\,|\, c(x')+x_3 d(x')\big)\nonumber\\
			&+ h^{\alpha+1} x_3 \left(\nabla' c(x')+\frac{x_3}{2}\nabla' d(x')\,\Big|\,0\right),\nonumber\\
			G_h(x)=&\begin{pmatrix}
				\begin{array}{c|c}
					\nabla' u(x')-x_3 (\nabla')^2 v(x') & -h^{-1}\nabla'v(x')^T\\
					\hline h^{-1}\nabla'v(x') & 0
				\end{array}
			\end{pmatrix}+ \big(0\,|\,0\,|\, c(x')+x_3 d(x')\big)\nonumber\\
			&+ h x_3 \left(\nabla' c(x')+\frac{x_3}{2}\nabla' d(x')\,\Big|\,0\right).\label{eq:explGh}
		\end{align}
		
		It is then immediate to check that \eqref{eq:convuva}--\eqref{eq:convuvc} hold true, while \eqref{eq:Ghconv2} is valid for $\sym G_h$ since one gets rid of the terms involving $h^{-1}$.
		
		By arguing as in \cite[Sections 6.1 and 6.2]{friesecke.james.mueller2}, we know that
		\begin{equation}\label{eq:lim1}
			\begin{aligned}
				\lim\limits_{h\to 0}\frac{1}{h^{2\alpha}}\int_\Omega W(\nabla_h y_h)\d x&=\frac 12 \int_\Omega Q_3(A(x')+x_3 B(x')) \d x\\&=\frac 12 \int_S Q_3(A(x')) \d x'+\frac{1}{24} \int_S Q_3(B(x')) \d x',
			\end{aligned}
		\end{equation}
		where
		\begin{equation*}
			A(x')=\begin{cases}\displaystyle
				e'(u(x'))+\frac 12 \nabla' v(x') \otimes\nabla' v(x')+\frac 12 |\nabla' v(x')|^2 e_3\otimes e_3+c(x')\otimes e_3, &\text{if }\alpha=2,\\
				e'(u(x'))+c(x')\otimes e_3, &\text{if }\alpha>2,
			\end{cases}
		\end{equation*}
		and 
		\begin{equation*}
			B(x')=-(\nabla')^2 v(x')+d(x')\otimes e_3.
		\end{equation*}
		
		In order to deal with the pressure term, we first use \eqref{eq:sviluppodet} obtaining
		\begin{equation}\label{eq:equal1}
			\frac{1}{h^\alpha}\int_\Omega (\det\nabla_h y_h(x)-1) \d x
			= \int_\Omega \tr G_h(x) \d x+h^\alpha\int_\Omega \iota_2(G_h(x))\d x+h^{2\alpha}\int_\Omega \det G_h(x)\d x.     
		\end{equation}
		
		By using the explicit expression \eqref{eq:explGh} of $G_h$, and recalling \eqref{eq:iota} and the bound $|\det F|\le |F|^3$, we deduce that
		\begin{align*}
			&\tr G_h(x)=\tr(e'(u(x'))-x_3 (\nabla')^2 v(x'))+c_3(x')+x_3 d_3(x')+ O(h),\\
			& \iota_2(G_h(x))= h^{-2}|\nabla' v(x')|^2+O(h^{-1})=h^{-2}\tr(\nabla' v(x')\otimes \nabla' v(x')) +O(h^{-1}),\\
			&\det G_h(x)=O(h^{-3}),
		\end{align*}
		where the Landau notation $O(\cdot)$ is uniform in $x\in \Omega$.
		
		Since $\alpha\ge 2$, passing to the limit in \eqref{eq:equal1} we thus obtain
		\begin{equation}\label{eq:lim2}
			\lim\limits_{h\to 0}\frac{1}{h^\alpha}\int_\Omega (\det\nabla_h y_h-1) \d x=\begin{cases}\displaystyle
				\int_S\Big(\tr(e'(u)+\nabla' v\otimes \nabla' v)+c_3\Big)\d x',&\text{if }\alpha=2,\\
				\displaystyle\int_S(\tr e'(u)+c_3)\d x',&\text{if }\alpha>2.
			\end{cases}
		\end{equation}
		
		Putting together \eqref{eq:lim1} and \eqref{eq:lim2} we finally infer
		\begin{equation}\label{eq:finalexpr}
			\lim\limits_{h\to 0}\frac{1}{h^{2\alpha}}\overline{\mc E}_h^{\alpha,\pi}(y_h)= \frac{1}{24} \int_S Q_3(-(\nabla')^2 v+d\otimes e_3) \d x'+\mc F^\alpha (u,v),
		\end{equation}
		where
		\begin{align*}
			\mc F^2 (u,v)=&\frac 12\int_S\left(Q_3\left(e'(u)+\frac 12\nabla' v\otimes \nabla' v+\frac 12 |\nabla' v|^2 e_3\otimes e_3+ c\otimes e_3\right)+2\pi c_3\right)\d x'\\
			&+\pi \int_S \left(\div ' u+|\nabla' v|^2\right) \d x',
		\end{align*}
		while for $\alpha>2$ one has
		\begin{align*}
			\mc F^\alpha (u,v)=\frac 12\int_S\left(Q_3\left(e'(u)+ c\otimes e_3\right)+2\pi c_3\right)\d x'+\pi \int_S \div ' u \d x'.
		\end{align*}
		
		We now choose
		\begin{equation}\label{eq:cd}
			\begin{aligned}
				&c(x'):=\begin{cases}
					\displaystyle -\frac 12 |\nabla' v(x')|^2 e_3+\overline c (x'),&\text{if }\alpha=2,\\
					\widetilde c (x'),&\text{if }\alpha>2,
				\end{cases}\\
				&d (x'):= \overline d (x'),
			\end{aligned}
		\end{equation}
		where $\overline c$, $\widetilde c$ and $\overline d$ satisfy (recall \eqref{def:Q2} and \eqref{def:Qpi})
		\begin{align*}
			&Q_2^\pi \left(\!e'(u(x')){+}\frac 12\nabla' v(x')\otimes \nabla' v(x')\!\right)\!=\! Q_3\left(\!e'(u(x')){+}\frac 12\nabla' v(x')\otimes \nabla' v(x'){+} \overline c(x')\otimes e_3\!\right){+}2\pi \overline c_3(x'),\\
			&Q_2^\pi(e'(u(x')))= Q_3(e'(u(x'))+\widetilde c(x')\otimes e_3)+2\pi \widetilde c_3(x'),\\
			& Q_2((\nabla')^2 v(x'))=Q_3(-(\nabla')^2 v(x')+\overline d(x')\otimes e_3).
		\end{align*}
		Since $u$ and $v$ are smooth, we observe that $\overline c$, $\widetilde c$ and $\overline d$, and so $c$ and $d$ in \eqref{eq:cd}, are smooth as well from $\overline S$ to $\R^3$.
		
		With this choice of $c$ and $d$, recalling \eqref{eq:Qrewriting}, we thus deduce that the right-hand side in \eqref{eq:finalexpr} coincides with
		\begin{equation*}
			\frac{\pi^2}{2}\kappa |S|+\begin{cases}
				\mc E^\pi_{\rm vK}(u,v),& \text{if }\alpha=2,\\
				\mc E^\pi_{\rm vK, lin}(u,v),& \text{if }\alpha>2,
			\end{cases}=\mc E^{\alpha,\pi}(u,v),
		\end{equation*}
		and so we conclude the proof.
	\end{proof}
	
	\bigskip
	
	\noindent\textbf{Data availability statement.} The manuscript has no associated data.
	
	\bigskip
	
	\noindent\textbf{Conflict of interest statement.} The authors have no competing interests to declare that are relevant to the content of this article.
	\bigskip
	
	\noindent\textbf{Acknowledgements.}
	The authors wish to thank  M. G. Mora for many fruitful discussions on the topic and D. Padilla-Garza for an inspiring conversation regarding the ansatz \eqref{eq:recovery} in the Kirchoff regime. M.K. was supported by the Czech Science Foundation 
	project 23-04766S and by the M\v{S}MT-WTZ project 8J24AT004.  He is indebted to the Department of Mathematics of the University of Pavia for hospitality during his stay there in 2024.  This research was partialy conducted  during MK's stay at the Tokyo Metropolitan University supported by the Invitational Fellowship S24117 awarded by the JSPS.  F.R. is member of GNAMPA (INdAM), and acknowledges its support through the INdAM-GNAMPA project 2025 \lq\lq DISCOVERIES\rq\rq (CUP E5324001950001).

\end{document}